\documentclass[pmlr,twocolumn,a4paper,10pt]{jmlr}
\usepackage{isipta2019}
\usepackage[british]{babel}
\usepackage{booktabs} % commands to create good-looking tables
\usepackage{tikz} % nice language for creating drawings and diagrams
\usepackage{pgfplots}
\usepackage{stanli} 
\usepackage{mathrsfs}
\usepackage[capitalise]{cleveref}
\usepackage{enumitem}
\usepackage{etoolbox}

\newtoggle{extendedversion}
\toggletrue{extendedversion}
\DeclarePairedDelimiter{\group}{(}{)}

\DeclarePairedDelimiter{\vset}{\{}{\}}

\DeclarePairedDelimiter{\norm}{\Vert}{\Vert}
\DeclarePairedDelimiter{\abs}{\vert}{\vert}

\newcommand{\cset}[3][]{\vset[#1]{#2\colon#3}}
\newcommand{\reals}{\mathbb{R}}
\newcommand{\extendedreals}{\overline{\mathbb{R}}\vphantom{\reals}}
\newcommand{\naturals}{\mathbb{N}}

\newcommand{\Pee}{\mathscr{P}}
\newcommand{\FF}{\mathscr{F}}
\newcommand{\E}{\mathrm{E}}

\newcommand{\Eo}{\underline{\E}}
\newcommand{\Eoh}{\hat{\Eo}\vphantom{E}}
\newcommand{\Eu}{\overline{\E}}
\DeclareMathOperator{\supp}{supp}
\newcommand{\dx}{\mathrm{d}x}
\newcommand{\ifandonlyif}{\Leftrightarrow}

\newcommand{\borels}{\mathscr{B}}
\newcommand{\muu}{\underline{\mu}}
\newcommand{\sigmu}{\underline{\sigma}}
\newcommand{\muo}{\overline{\mu}}
\newcommand{\sigmo}{\overline{\sigma}}

\newcommand{\asconvergesto}{\xrightarrow{\text{a.s.}}}

\DeclareMathOperator{\id}{\mathrm{id}}
\newenvironment{proofof}[1]{\par\noindent{\bfseries\upshape Proof of #1\ }}{\jmlrQED}

\jmlrworkshop{ISIPTA 2019}

\title[Monte Carlo Estimation for Imprecise Probabilities: Basic Properties]{% [Short title] can be omitted if title is short enough
  Monte Carlo Estimation for Imprecise Probabilities:\\ % tilde needed to get space in short title
  Basic Properties
}

\author{% Authors are listed one per line and grouped per affiliation
  \Name{Arne {Decadt}}\Email{arne.decadt@ugent.be}\\ 
  \Name{Gert {de Cooman}}\Email{gert.decooman@ugent.be}\\
  \Name{Jasper {De Bock}}\Email{jasper.debock@ugent.be}\\
  \addr ELIS -- FL\hspace{.055555em}ip, Ghent University, Belgium
}

\begin{document}
\maketitle

\begin{abstract}% No length restrictions, but keep things reasonable
We describe Monte Carlo methods for estimating lower envelopes of expectations of real random variables.
We prove that the estimation bias is negative and that its absolute value shrinks with increasing sample size. 
We discuss fairly practical techniques for proving strong consistency of the estimators and use these to prove the consistency of an example in the literature.
We also provide an example where there is no consistency.
\end{abstract}

\section{Introduction}\label{sec:intro}
Monte Carlo simulation is a technique widely used for obtaining numerical estimates for in principle anything that can be represented as a sum or an integral.
Multiple research fields have proposed to extend the Monte Carlo method towards solving optimisation problems. 
Both the fields of stochastic optimisation and empirical processes have for example studied the properties of such generalised Monte Carlo estimators, including their bias \cite{Bayraksan2006}, their consistency and a version of the central limit theorem \cite{shapiro1991asymptotic,van1996weak} for them. 
%On the other hand, the field of imprecise probabilities has thus far been mostly concerned with trying to increase the convergence speed of importance sampling for finding extremal values \cite{Troffaes2018}.
In the field of imprecise probabilities, Monte Carlo based methods have also received considerable attention~\cite{faes2018failure,fetz2016imprecise,lazar2003exploring,oberguggenberger2009classical,schobi2017structural,Troffaes2018,zhang2012interval,zhang2018quantification}, but with the exception of the discussion in~\cite{Troffaes2018}, the theoretical properties of the corresponding estimators have not been investigated in much detail. We here address this by studying the bias and consistency of a large class of imprecise Monte-Carlo based estimators, including the important special case of imprecise Monte-Carlo estimators that are based on importance sampling~\cite{fetz2016imprecise,Troffaes2018,zhang2012interval,zhang2018quantification}.
%Troffaes~\cite{Troffaes2018} was not the first author in the field to consider importance sampling for imprecise Monte Carlo methods.
%Lazar~\cite{lazar2003exploring} proposed similar ideas in the context of Markov Chain Monte Carlo methods, and Zhang~\cite{zhang2012interval} discusses them in the context of the present paper, but tailored to the specific case of failure probability estimation.
%Zhang and Shields \cite{zhang2018quantification} consider the problem of finding the optimal density for importance sampling for the minimization over multiple distributions.
% For a larger arsenal of methods, that lie outside the scop of this paper, we refer to Oberguggenberger~\cite{oberguggenberger2009classical}, who gives an overview of alternative methods, specifically for sensitivity analysis.
% Faes~\cite{faes2018failure} presents two surrogate models and compares more advanced sampling methods for them, in particular Line Sampling and SubSet sampling.
% A method involving Sobol' indices is considered by Schöbi~\cite{schobi2017structural}.

%\cite{alvarez2017tighter}

The optimisation or extremum problem we typically encounter in the imprecise probabilities field is calculating the lower expectation \(\Eo(f(X))\)---or the upper expectation \(\Eu(f(X))=-\Eo(-f(X))\)---of a Borel measurable map~\(f\) of a real random variable \(X\), whose distribution \(P\) is only known to belong to a set \(\Pee\).\footnote{More generally, \(X\) could be a real random vector, and~\(f\) a vector-valued map, but we will only consider the scalar case here.}
To find this lower expectation 
\begin{equation*}
\Eo(f(X))=\inf_{P\in\Pee}\E^P(f),
\end{equation*}
we have to minimise the expected value \(\E^P(f)\)---a sum or an integral that would typically be estimated using Monte Carlo simulation---over a set of distributions \(\Pee\).
A fairly straightforward---but naive---method for estimating this infimum consists in choosing a finite subset \(\Pee'\) of \(\Pee\), taking independent samples for each probability measure \(P\in\Pee'\), using them to get to a Monte Carlo estimate for \(\E^P(f)\), and finally finding the minimum of the Monte Carlo estimates.
The problem with this method is that the larger the set of probabilities~\(\Pee'\), the larger the bias it will tend to produce, as we argue in Section~\ref{sec:imprecise:monte:carlo}.

To reduce this bias, we would typically want to introduce correlation between the estimates of \(\E^P(f)\) for different \(P\in\Pee\).
One way to do this is to sample from a single distribution only, and to use these samples to calculate the corresponding samples for the distributions in the set \(\Pee\).
We will therefore in this paper study estimators of the form
\begin{equation}\label{eq:standardform:1}
\Eoh_n^{\Pee}(f)
=\inf_{t\in T}\frac1n\sum_{k=1}^{n}f_t(X^P_k),
\end{equation}
where we take independent samples \(X^P_k\) from a single distribution \(P\) and transform the results into samples for other distributions, parametrised by a set \(T\)---this transformation, as we shall see further on, accounts for the use of the maps~\(f_t\) rather than the map~\(f\) in the estimator~\eqref{eq:standardform:1}.
A word on terminology is in order here.
We will be careful in this paper to distinguish between functions and maps on some domain; maps are always defined on all of the domain, whereas functions may, but need not be, defined only on a part of the domain.

The decomposition of the problem into a classical random variable \(X^P\) and a parametric part \(T\) is essential for our results, and is also present in the literature \cite{alvarez2017tighter,Troffaes2018}.
We show how to prove consistency for this fairly general type of estimators, and we give sufficient conditions for consistency that are relatively easy to check.

The plan of the paper is as follows.
First, in Section~\ref{sec:monte:carlo:estimation}, we describe the theoretical setting of the discussion, and explain in some detail why we use estimators of the general form~\eqref{eq:standardform:1}.
We discuss two particular methods for generating samples of sufficient generality, where the resulting estimator has this general form.
In Section~\ref{sec:estimator:properties} we study a number of statistical properties of estimators of type~\eqref{eq:standardform:1}: their bias and consistency.
To show that the results we derive are practically useful, we include a discussion of two examples, in Section~\ref{sec:examples}.
In the first example, we show how to actually prove the consistency of an estimator encountered in the literature~\cite{fetz2016imprecise}.
The second example is a case where there is no consistency, provably.
In the Conclusion, we highlight our main findings, and provide pointers to future work.

\iftoggle{extendedversion}
{In order not to interrupt the flow of the discussion too much, we have moved all proofs to the Appendix.}
{Proofs of our results are available in an extended arXiv version \cite{extpaper}.}

\section{Monte Carlo Estimation}\label{sec:monte:carlo:estimation}
% In this section we will introduce a general imprecise Monte Carlo estimator, whose properties can later be applied to more advanced Monte Carlo methods.

\subsection{Notation and Theoretical Setting}\label{sec:notation:and:setting}
Let us begin by sketching the theoretical background, and fixing the basic notation.

We consider the measurable space \((\reals,\borels)\), where \(\borels\) is the Borel \(\sigma\)-field on the set \(\reals\) of all reals.
For any given probability measure \(P\) on \((\reals,\borels)\), we denote by \(X^P\coloneqq\id\)---the identity map on \(\reals\)---the corresponding real random variable on the probability space \((\reals,\borels,P)\), and we say that \(X^P\) has distribution~\(P\). 
The real random variable \(X^P\) has a right-continuous \emph{cdf} (cumulative distribution function) \(F\colon\reals\to[0,1]\), defined by \(F(x)\coloneqq P(X\leq x)=P((-\infty,x])\) for all \(x\in\reals\). 
\(X^P\) will be called continuous if its distribution has a density---is absolutely continuous with respect to the Lebesgue measure on the measurable space \((\reals,\borels)\).

For the purposes of sampling, the single probability space \((\reals,\borels,P)\) is extended to the canonical (independent) product space \((\reals^\infty,\borels^\infty,P^\infty)\), as for instance described in \cite{dudley_2002,van1996weak}. 
We now identify the countable infinity of real random variables \(X^P_k\) on \((\reals^\infty,\borels^\infty,P^\infty)\), each of which is the projection map from the Cartesian product \(\reals^\infty\) to its \(k\)-th component space \(\reals\).
The infinite sequence \((X^P_1,X^P_2,\dots)\) of real random variables \(X^P_k\) is then also a random variable on the probability space \((\reals^\infty,\borels^\infty,P^\infty)\), and is also denoted by \(X^P_{1:\infty}\); a finite subsequence \((X^P_1,\dots,X^P_n)\) is denoted by \(X^P_{1:n}\).
We denote by \(\E^\infty\) the expectation operator associated with this product probability space.

We also consider a parameter set \(T\) that parametrises a set of probability measures \(\Pee\coloneqq\cset{P_t}{t\in T}\) such that each~\(P_t\) leads to a probability space \((\reals,\borels,P_t)\) with generic real random variables \(X^{P_t}\). 
The corresponding probability space for its sequences of samples is \((\reals^\infty,\borels^\infty,P_t^\infty)\).
The product space for all these sequences together, so the product of all the \((\reals^\infty,\borels^\infty,P_t^\infty)\) over all $t\in T$, is denoted by \((\reals_T^\infty,\borels_T^\infty,P_T^\infty)\), with associated expectation operator~$\E_T^\infty$.

\subsection{Imprecise Monte Carlo Estimation}\label{sec:imprecise:monte:carlo}
As mentioned in the Introduction, in a so-called imprecise Monte Carlo simulation we want to estimate the lower expectation \(\Eo(f(X))=\inf_{P\in\Pee}\E^P(f)\) of some Borel measurable real-valued map~\(f\) of a real random variable~\(X\), whose distribution~\(P\) is only known to belong to a set~\(\Pee\).

To get some grip on this problem, let us first look at how we would typically proceed in the so-called classical case, where \(\Pee=\{P\}\) is a singleton.

We now consider the real random variables \(X^P_k\).
We also consider an `\(P\)-integrable' Borel measurable map \(f\colon\reals\to\reals\), where `\emph{\(P\)-integrable}' means that \(\E^P(\abs{f})<+\infty\). 
We define the sample mean estimator \(\hat{\E}_n^P\) for the expectation \(\E^P(f)\) as the real random variable
\begin{equation*}
\hat{\E}_n^P(f)
\coloneqq\frac1n\sum_{k=1}^{n}f(X_k^{P}).
\end{equation*}
This real random variable on \((\reals^\infty,\borels^\infty,P^\infty)\) is clearly a map of the real random variables \(X_1^{P},\ldots,X_n^P\), but we will suppress their appearance in the notation for \(\hat{\E}_n^P(f)\).
It is a classical result that \(\hat{\E}_n^P\) is an unbiased estimator for \(\E^P\group{f}\), meaning that \(\E^{\infty}\group[\big]{\hat{\E}_n^P\group{f}}=\E^P\group{f}\).

In an imprecise probabilities setting, we consider a set of probability measures \(\Pee\) on the measurable space \((\reals,\borels)\).
As explained in \cref{sec:notation:and:setting}, we will typically use a set \(T\) to index the probability measures in \(\Pee\), so \(\Pee=\cset{P_t}{t\in T}\).
Here too, the Borel measurable map \(f\colon\reals\to\reals\) will be assumed throughout to be `\emph{\(\Pee\)-integrable}', which we from now on take to mean that $\Eu\vphantom{\E}^{\Pee}(\abs{f})\coloneqq\sup_{t\in T}\E^{P_t}(\abs{f})<+\infty$.
It is easy to see that this implies that \(f\) is then \(P\)-integrable for every \(P\in\Pee\), and that
\begin{equation*}
 -\infty
 <\Eo^{\Pee}(-\abs{f})
 \leq\Eo^{\Pee}(f)
 \leq\Eu\vphantom{\E}^{\Pee}(f)
 \leq\Eu\vphantom{\E}^{\Pee}(\abs{f})
 <+\infty.
\end{equation*}
In a naive first attempt, one might consider the following estimator for \(\Eo^{\Pee}(f)\):
\begin{equation}\label{eq:deflowprev}
\Eoh_n^{\Pee,\mathrm{n}}(f)
\coloneqq\inf_{t\in T}\hat{\E}_n^{P_t}(f)
=\inf_{t\in T}\frac1n\sum_{k=1}^{n}f(X_k^{P_t}).
\end{equation}
This is a(n extended) real map on a suitably defined product space, such as for instance~\((\reals_T^\infty,\borels_T^\infty,P_T^\infty)\).
In practice, however, this type of estimator has a number of undesirable properties.
First of all, it is computationally inefficient.
Since all currently available sampling methods sample with respect to a single distribution, taking an infimum would mean that we need to take samples for each distribution in at least a representative enough part of \(\Pee\).
Suppose we choose \(m\) probability measures in \(\Pee\), for each of which we take a sample of size \(n\), then we have to take \(mn\) samples in total.
If we increase the number of sampled distributions in \(\Pee\) by one, say, to get a better approximation, then the number of samples to take will increase by~\(n\).

Secondly, adding new distributions in \(\Pee\) to the sampling procedure---in an attempt to get to a better approximation---will not only increase the computational cost, but it can produce the adverse effect of increasing the absolute value of the bias of the resulting estimator~\eqref{eq:deflowprev}. 
To see why this is so, we give an example for a simplified case.
Consider a set of probability measures \(\Pee=\cset{P_t}{t\in T}\) that all associate the same expectation with a given map~\(f\), so we can use \(\E^P(f)\) to denote this common expectation: \(\E^{P}(f)=\E^{P_{t}}(f)\) for all~\(t\in T\).
% To see why, we consider the following simplified argument.
Take two finite disjoint subsets \(H\) and \(\Delta H\) of \(T\) and let us consider the approximations---with obvious notations---\(\Eoh_n^{H,\mathrm{n}}(f)\) and \(\Eoh_n^{H\cup\Delta H,\mathrm{n}}(f)\) of \(\Eoh_n^{\Pee,\mathrm{n}}(f)\), where we clearly expect, conceptually, \(\Eoh_n^{H\cup\Delta H,\mathrm{n}}(f)\) to be the better approximation.
However, this turns out not to be the case, since, again with obvious notations,
\begin{equation*}
\Eoh_n^{H\cup\Delta H,\mathrm{n}}(f)
=\min\vset[\Big]{\Eoh_n^{H,\mathrm{n}}(f),\Eoh_n^{\Delta H,\mathrm{n}}(f)},
\end{equation*}
so it follows that\footnote{The expectation \(\E_T^\infty\) is taken in a suitable product space, such as \((\reals_T^\infty,\borels_T^\infty,P_T^\infty)\).}
\begin{equation*}
\E_T^\infty\group[\Big]{\Eoh_n^{H\cup\Delta H,\mathrm{n}}(f)}
\leqslant
\E_T^\infty\group[\Big]{\Eoh_n^{H,\mathrm{n}}(f)}
\end{equation*}
and, with a similar argument,
\begin{align*}
\E_T^\infty\group[\Big]{\Eoh_n^{H,\mathrm{n}}(f)}
\leqslant\inf_{t\in H}\E_T^\infty\group[\Big]{\hat{\E}_n^{P_t}(f)}
=\E^P(f)=\Eo^{\Pee}\group{f}.
\end{align*}
We see that the bias of our estimator, which is already non-positive, can only become more negative.

The proposed solution to the computational efficiency problem consists in sampling from a single distribution \(P\) only, and then to reuse the samples \((X^P_1,\dots,X^P_n)\) in some way, and to transform them into the desired samples for the distributions \(P_t,t\in T\).
This will typically also produce correlation between the estimators \(\hat{\E}^{P_t}_n(f)\) for different \(t\in T\), and so hopefully reduce the bias problem as well.

A first way to achieve this is the method of \emph{inverse transform sampling}, introduced in \cref{sec:inverse:transform:sampling}, where we take samples from a uniform distribution on the open real unit interval~\((0,1)\), then use the quantile functions of the distributions \(P_t,t\in T\) to transform these samples into samples for the real random variables~\(X^{P_t}\), and use these transformed samples to approximate the lower expectation. 
A second implementation of this same idea is \emph{importance sampling}, discussed in \cref{sec:importance:sampling}, and used often when the real random variables are continuous.
Here, the reuse is implemented by using weight factors, dependent on the \(P_t\), to rescale the terms in the sample mean.

Interestingly, the exact form of the estimator does not matter much for the theoretical analysis further on, as long as there is a sample \(X^P_{1:\infty}\) taken from some probability measure \(P\), and Borel measurable maps \(f_t\colon\reals\to\reals\) such that \(\E^{P_t}(f)=\E^P(f_t)\) and also \(\E^{P_t}(\abs{f})=\E^P(\abs{f_t})\), for every \(t\in T\).
Since it follows from its assumed \(\Pee\)-integrability that $f$ is also $P_t$-integrable for all $t\in T$, this implies that $f_t$ is $P$-integrable for all $t\in T$.
We will show in the following sections how to find these measurable maps~\(f_t\) in both of the above-mentioned cases.

The estimator for the lower expectation can then be written as
\begin{equation}\label{eq:standardform}
\Eoh_n^{\Pee}(f)
\coloneqq
\inf_{t\in T}\frac1n \sum_{k=1}^{n} f_t(X^P_k)
=\inf_{t\in T}\hat{\E}_n^{P}(f_t).
\end{equation}
This estimator is an extended real-valued map on \(\reals^{\infty}\) that, due to the presence of the infimum, is not necessarily measurable with respect to the \(\sigma\)-field \(\borels^{\infty}\). 
If \(T\) is a compact subset of $\reals^n$ and the maps~\(f_t\) are continuous in \(t\), measurability will still be guaranteed.
This is the case considered by Troffaes in~\cite{Troffaes2018}.
One of the aims of the present paper is to go beyond this special case, and to allow for the fact that \(\Eoh_n^{\Pee}(f)\) need not be measurable in general.
In order to deal with this complication mathematically, we need to extend the expectation operator \(\E^{\infty}\) associated with the probability measure \(P^{\infty}\)---and defined on the measurable maps---to its (inner and) outer expectation(s) \cite{van1996weak} defined on arbitrary extended real-valued maps $h$:\footnote{See also \cite{koenig1997,troffaes2013:lp} for similar ideas.}
\begin{equation*}
\begin{aligned}
&\Eu^{\infty}(h)\coloneqq\\
&\inf\cset{\E^{\infty}(g)}
{h\leqslant g\in\extendedreals^{\reals^\infty}\text{ measurable and }\E^{\infty}(g)\text{ exists}},
\end{aligned}
\end{equation*}
where \(\extendedreals^{\reals^\infty}\) denotes the set of all maps from \(\reals^\infty\) to \(\extendedreals\), where `\(\E^{\infty}(g)\) exists' is taken to mean that at least one of \(\E^{\infty}(g^+)\) and \(\E^{\infty}(g^-)\) is finite, and where we introduced the notations \(g^+\coloneqq\max\vset{g,0}\) and \(g^-\coloneqq\max\vset{-g,0}\).
The set \(\extendedreals\coloneqq\reals\cup\vset{-\infty,+\infty}\) is the set of the extended real numbers.
The inner expectation \(\Eo^{\infty}\) is defined by \(\Eo^{\infty}(h)\coloneqq-\Eu^{\infty}(-h)\).
Interestingly, the following lemma guarantees the existence of a so-called \emph{minimal measurable cover}, which we will denote with a superscript star.\footnote{The formulation in~\cite[Lemma~1.2.1]{van1996weak} is slightly off, as a simple check of the proof attests: the `\(g\geqslant f\) a.s.' there must be replaced by `\(g\geqslant f\)', as we do here.}

\begin{lemma}[{\protect\cite[Lemma~1.2.1]{van1996weak}}]\label{lem:mescov}
For any map~\(f\colon\reals^\infty\to\extendedreals\), there is a measurable map~\(f^*:\reals^\infty\to\extendedreals\) with
\begin{enumerate}[label=\upshape(\roman*),leftmargin=*]
\item \(f^* \geqslant f\);
\item for any measurable \(g\colon\reals^\infty\to\extendedreals\) such that \(g\geqslant f\), it holds that \(g\geqslant f^*\) a.s.
\end{enumerate}
For any such \(f^*\) it holds that\/ \(\Eu^\infty (f)=\E^\infty (f^*) \) provided that \(\E^\infty(f^*)\) exists, which is certainly true if \(\Eu^\infty(f)<+\infty\) .
\end{lemma}

\subsection{Inverse Transform Sampling}\label{sec:inverse:transform:sampling}
For univariate distributions, the problem of sampling from a given distribution has many known solutions. 
One of the most commonly used techniques is inverse transform sampling, which we now briefly describe.
Many known algorithms can generate pseudo-random numbers in the open real unit interval~\((0,1)\), i.e.~produce samples from the uniform distribution~\(P\) on \((0,1)\).
We use the generic notation~\(U\) for a real random variable that is uniformly distributed on~\((0,1)\).
We define the \emph{quantile function}, or \emph{pseudo-inverse}, \(F^{\dagger}\colon[0,1]\to\extendedreals\), of a given cdf \(F\colon\reals\to[0,1]\) by 
\begin{equation}\label{cumin}
F^{\dagger}(x)
\coloneqq\inf\cset{ y\in\reals}{x \leqslant F(y)},\quad x\in[0,1]. %\reals.
\end{equation}
We prove in the Appendix \iftoggle{extendedversion}{}{of the extended arXiv version~\cite{extpaper}} that \(F^{\dagger}(U)\) is a real random variable with cdf \(F\).
This means we can use samples from the uniform distribution on \((0,1)\) to generate samples from a distribution with cdf \(F\).\footnote{This technique can be extended to the independent multivariate case by simple repetition. The dependent case is more involved, but still feasible.}
% Given a general cdf with \(k\) variables \(F_{X_1,\ldots, X_k}\colon\reals^k\to [0,1]\), we denote the marginal cdf of a given random variable \(X_n\) by \(F_{X_n}\) and the marginal cdf of \(X_n\) conditional on \(X_{m}\) by \(F_{X_n|X_m}\).
% Inverse transform sampling uses the relation
% \begin{equation}
% \begin{aligned}
% &F_{X_1,\ldots, X_k}(x_1,\ldots,x_n)\\&=F_{X_1}(x_1) F_{X_2|X_1}(x_2 | x_1) F_{X_3|X_1,X_2}(x_3 | x_1,x_2) \ldots\\ & \qquad F_{X_k|X_1, \ldots, X_{k-1}}(x_k | x_1,\ldots,x_{k-1}),
% \end{aligned}
% \end{equation}
% based on Bayes' Rule for cdf.
% First \(X_1\) is generated from a uniform variable \(U_1\) as \(X_1=F^{\dagger}
% _{X_1}(U_1)\). 
% Then, using the sample value of \(X_1\), \(X_2\) is generated from a uniform variable \(U_2\) as \(X_2=F^{\dagger}
% _{X_2|X_1}(U_2 | x_1)\).
% This procedure is repeated until \(X_{1:k}\) are all sampled. 
% The problem with this method, is that the conditional cdf depends on the obtained value of \(x_1\) and needs to be recalculated every time.
% The fact that increasing the dimension, makes the calculations harder is called the `curse of dimensionality'.
This is exactly what we need to find the maps~\(f_t\) in the previous section.
Indeed, let \(F_{P_t}\) denote the cdf of the distribution \(P_t,t\in T\), then with samples \(U_{1:\infty}\) from the uniform distribution~\(P\) on~\((0,1)\) we can construct the following estimator:
\begin{equation}\label{eq:lowprev}
\Eoh_n^{\Pee,U}(f)
\coloneqq\inf_{t\in T}\frac1n\sum_{k=1}^{n}f(F_{P_t}^{\dagger}(U_k)).
\end{equation}
For every \(t\in T\), we now consider the function \(f_t\coloneqq f\circ F_{P_t}^{\dagger}\), which is strictly speaking only defined---and therefore a map---on the open real unit interval~\((0,1)\).
But since this interval has measure one for the uniform distribution~\(P\) on~\((0,1)\), there is a collection of Borel measurable real maps that extend~\(f_t\) to all of~\(\reals\) and that are all almost surely equal with respect to the uniform distribution~\(P\) on~\((0,1)\).
We will pick one such member of this collection, and also denote it by~\(f_t\).
This is the measurable map~\(f_t\) that we promised to identify for \cref{eq:standardform}: it indeed follows from the discussion above that \(\E^{U}\group{f_t}=\E^{P_t}\group{f}\), and since \(\abs{f_t}=\abs{f}\circ F_{P_t}^{\dagger}\), that also \(\E^{U}\group{\abs{f_t}}=\E^{P_t}\group{\abs{f}}\).
The estimator in \cref{eq:lowprev,eq:standardform} does not depend on the actual choice of the extension~\(f_t\), as all extensions coincide on $(0,1)$.

%The practical computation of the optimisation can be tricky.

% \opletten{Hiervoor moet \(f\) wel gedefinieerd zijn op \(\extendedreals\)! Dit probleem kan wellicht worden opgelost door de uniforme op \((0,1)\) te nemen?}

% %\vraag{Is dit nodig: In some cases computations can be easier by sampling from other central distributions than the uniform distribution.
% If the new central random variable is \(Y\) with cdf \(F_Y\), then \(X\) needs to be replaced by \(F_Y(Y)\).
% This works under continuity conditions, for example it is not possible to generate every number from the interval \([0,1]\) with a single coin flip.
% A sufficient condition is that \(Y\) is continuous.}
% Dit bewijs ik dan misschien wel in de appendix ofzo

\subsection{Importance sampling} \label{sec:importance:sampling}

We now give a very brief account of Troffaes's~\cite{Troffaes2018} discussion on how to use importance sampling to do Monte Carlo simulations of lower expectations in an imprecise probabilities setting.
We consider a set \(\Pee\) of probability measures that are absolutely continuous with respect to the Lebesgue measure.
For every probability measure \(P_t\in \Pee\), we denote its corresponding density by \(p_t\). 
For importance sampling we also need (another), so-called \emph{central}, probability measure \(P\) that is also absolutely continuous, and has density \(p\).
This \(P\) need not be a member of \(\Pee\), but we do assume that its support \(\supp p\) is a superset of the supports \(\supp p_t\) for all \(P_t\in \Pee\).
Importance sampling is based on the fact that the lower expectation can then be written as\footnote{Here and in what follows, we use \(\int_A h(x)\dx\) to denote the Lebesgue integral of the measurable map~\(h\) associated with the Lebesgue measure on the measurable subset~\(A\) of the reals.} 
\begin{equation*}
\Eo^{\Pee}(f)=\inf_{t\in T}\int_{\supp p}f(x)\frac{p_t(x)}{p(x)}p(x)\dx.
\end{equation*}
Inspired by the classical Monte Carlo estimator, we define the importance sampling estimator, as
\begin{equation}\label{eq:standard:form:importance:sampling}
\Eoh_n^{\Pee/P}(f)
\coloneqq\inf_{t\in T}\frac1n\sum_{k=1}^{n}f(X_k^P)\frac{p_t(X_k^P)}{p(X_k^P)}
=\inf_{t\in T}\frac1n\sum_{k=1}^{n}f_t(X_k^P),
\end{equation}
where we defined, for any~\(t\in T\), the real function~\(f_t\) by
\begin{equation}\label{eq:functions:for:importance:sampling}
f_t\coloneqq 
f\frac{p_t}{p}.
\end{equation}
Observe that~\(f_t\) is only defined on the measurable set~\(\supp p\), which has measure one for the central distribution~\(P\).
Here too, \(f_t\) can be extended to a collection of Borel measurable real maps defined on all of~\(\reals\) that are all almost surely equal with respect to~\(P\).  
We will again pick one member of this equivalence class, and also denote it by~\(f_t\).
Then indeed also here
\begin{align*}
\E^{P_t}(f)
&=\int_{\supp p_t} f(x) p_t(x)\dx\\
&=\int_{\supp p} f(x) \frac{p_t(x)}{p(x)}p(x)\dx
=\E^P(f_t).
\end{align*}
A similar argument shows that also \(\E^{P_t}(\abs{f})=\E^P(\abs{f_t})\).
Observe that the estimator in \cref{eq:standardform,eq:standard:form:importance:sampling} does not depend on the actual choice of the extension~\(f_t\), as all extensions coincide on the range of the~\(X^P_k\).

In general, finding a good central probability measure~\(P\) is an important but difficult problem. 
Troffaes~\cite{Troffaes2018} proposes to use iterated importance sampling as a potential solution.

\section{Estimator Properties}\label{sec:estimator:properties}
Now that the most important concepts have been introduced, we investigate the properties of the estimators in the standard form~\eqref{eq:standardform}; specialisations to the case of importance sampling with be discussed in more detail in \cref{sec:consistency:for:importance:sampling}. 
For all theorems and definitions below we consider the measurable space \((\reals,\borels)\) and a set of probability measures \(\Pee=\cset{P_t}{t\in T}\) such that, for every \(t\in T\), \((\reals,\borels,P_t)\) constitutes a probability space.
We also consider a probability measure~\(P\) on~\((\reals,\borels)\) that is not necessarily in~\(\Pee\) but does constitute a probability space~\((\reals,\borels,P)\), as well as some \(\Pee\)-integrable Borel measurable map~\(f\colon\reals\to\reals\). 
This implies that \(\Eu^{\Pee}(\abs{f})<+\infty\) and therefore, that \(f\) is \(P_t\)-integrable for all \(t\in T\).
On the product probability space~\((\reals^{\infty},\borels^{\infty},P^{\infty})\), we define the sequence of real random variables~\(X^{P}_{1:\infty}\).
Finally, we consider a class of Borel measurable maps \(\FF\coloneqq\cset{f_t\in\reals^\reals}{t\in T}\), such that for every \(t\in T\), maps of the type \(f_t\circ X^{P}_k\) are measurable.
Alternatively and equivalently, we can consider the single map~\(f_T\colon\reals\times T\to\reals\colon(x,t)\mapsto f_t(x)\).
We assume that the maps~\(f_t\) are chosen in such a way that \(\E^P(f_t)=\E^{P_t}(f)\) and \(\E^P(\abs{f_t})=\E^{P_t}(\abs{f})\) for all \(t\in T\), so as a consequence, \(f_t\) is \(P\)-integrable for all \(t\in T\).
Finally, \(n\) and~\(m\) will typically denote positive integers.

\subsection{Bias}\label{sec:estimator:bias}
Troffaes~\cite{Troffaes2018} correctly states that the literature on imprecise probabilities has paid little attention to the bias of imprecise Monte Carlo estimators. 
Fortunately, much more can be found in the literature on stochastic programs. 
The following result is, for instance, similar to Theorems~1 and~2 in~\cite{mak1999monte}.
The difference is that our estimators involve infima, rather than minima, so we need to use inner and outer expectations in the formulation.
Our result generalises the lower bound of~\cite[Theorem~1]{Troffaes2018} and can also be regarded as an extension (from $1$ sample to $n$ samples) of the one-dimensional version of~\cite[Theorem~3]{alvarez2017tighter}. %new 

\begin{theorem}[Bias]\label{th:decrimp}\\
Assume that \(-\infty<\Eo^{\infty}\group*{\inf_{t\in T}f_t\circ X_1^P}\),\footnote{It would make sense, for any map \(h\colon\reals\to\extendedreals\), to use the notation \(\Eo^P(h)\coloneqq\Eo^\infty\group*{h\circ X_1^P}\). The condition could then also be rewritten as \(-\infty<\Eo^{P}\group*{\inf_{t\in T}f_t}\).} then
\begin{equation*}
\Eo^{\infty}\group[\Big]{\Eoh_{n-1}^{\Pee}(f)}
\leqslant 
\Eo^{\infty}\group[\Big]{\Eoh_n^{\Pee}(f)}
\leqslant 
\Eu^{\infty}\group[\Big]{\Eoh_n^{\Pee}(f)}
\leqslant 
\Eo^{\Pee}(f).
\end{equation*}
\end{theorem}
This result already guarantees that the bias 
\begin{equation*}
\Eo^{\infty}\group[\Big]{\Eoh_{n}^{\Pee}(f)-\Eo^{\Pee}(f)}
\end{equation*} 
with respect to \(\Eo^{\infty}\) converges as $n\to+\infty$, since then
\begin{equation*}
\Eo^{\infty}\group[\Big]{\Eoh_{n-1}^{\Pee}(f)-\Eo^{\Pee}(f)}
\leqslant\Eo^{\infty}\group[\Big]{\Eoh_n^{\Pee}(f)-\Eo^{\Pee}(f)}
\leqslant0,
\end{equation*}
so the bias is non-decreasing and bounded above by zero.
Also the bias
\begin{equation*}
\Eu^{\infty}\group[\Big]{\Eoh_{n}^{\Pee}(f)-\Eo^{\Pee}(f)}\leqslant 0
\end{equation*} 
with respect to \(\Eu^{\infty}\) is bounded from above by zero.
Next, we turn to conditions that guarantee that our estimator is consistent, and therefore also asymptotically unbiased.

\subsection{Consistency}\label{sec:estimator:consistency}

To prove consistency we need to specify the way in which we want our estimator to converge.
We denote almost sure convergence to zero of a sequence of measurable maps \(g_n\colon\reals^\infty\to\extendedreals\) with respect to the probability measure \(P^\infty\) as \(\smash{g_n\asconvergesto 0}\).
Because the estimator \(\Eoh_n^{\Pee}\) in~\eqref{eq:standardform} is not necessarily measurable, due to the presence of an infimum in its definition, we need to extend the definition of almost sure convergence to non-measurable maps.

\begin{definition}[{\protect\cite[Definition 1.9.1]{van1996weak}}]
A sequence of maps \(h_n\colon\reals^\infty\to\extendedreals\) is said to converge (outer) almost surely to zero if, for any natural \(n\in\naturals\) there is some measurable cover \(\abs{h_{n}}^*\) of $\abs{h_n}$ such that \(\abs{h}^{*}_{n} \asconvergesto 0\), meaning that
\begin{equation*}
P^\infty\group[\Big]{\lim_{n\to+\infty}\abs{h_n}^*=0}=1.
\end{equation*}
\end{definition} 

Troffaes~\cite{Troffaes2018} proves that if the class of measurable maps \(\FF=\cset{f_t}{t\in T}\), for which \(\E^{P}(f_t)=\E^{P_t}(f)\), is a Glivenko--Cantelli class for \(P\), then the estimator~\eqref{eq:standardform} is consistent. 
We will derive a similar, but slightly stronger, result, using a more general definition of a Glivenko--Cantelli class, in order to deal with potential non-measurability.

\begin{definition}[{\protect\cite[Section 2.1]{van1996weak}}]\label{def:glivenko:cantelli:class}
A class~\(\Phi\) of Borel measurable and \(P\)-integrable maps $\phi$ is called a (strong) Glivenko--Cantelli class for the probability measure~\(P\) whenever
\begin{equation}
\group[\bigg]{\sup_{\phi\in\Phi}\abs[\Big]{\hat{\E}^P_n(\phi)-\E^P(\phi)}}^*\asconvergesto0.
\end{equation}
\end{definition}
\noindent Our next result then generalises Theorem~6 in~\cite{Troffaes2018}.

\begin{theorem}[Consistency]\label{th:consistent}
If \(\FF\) is a (strong) Glivenko--Cantelli class for the probability measure~\(P\), then the estimator\/ \(\Eoh^{\Pee}(f)\) for\/~\(\Eo^{\Pee}(f)\), as defined in \cref{eq:standardform}, is (strongly) consistent, meaning that
\begin{equation*}
\abs[\Big]{\Eoh^{\Pee}_n(f)-\Eo^{\Pee}(f)}^*\asconvergesto0.
\end{equation*}
\end{theorem}

It is not easy to check the criterion in Definition~\ref{def:glivenko:cantelli:class} in practice.
Pollard~\cite{pollard2012convergence} provides sufficient conditions for being a Glivenko--Cantelli class that may 
still not be practically useful immediately, but do constitute the starting point for more useful tools.
Let us take a look at his more direct approach to proving that a set is a Glivenko--Cantelli set, using the technique called \emph{bracketing}.

\begin{theorem}[{\protect\cite[II.2 Thm.~2]{pollard2012convergence}}]\label{th:brack}
Consider, for each \(\epsilon>0\), a class \(\Phi_{\epsilon}\coloneqq\cset{(\underline{\phi}_\epsilon,\overline{\phi}_\epsilon)}{\phi\in\Phi}\) containing lower and upper bounds---brackets---for each \(\phi\in\Phi\), where \(\underline{\phi}_\epsilon,\overline{\phi}_\epsilon\colon\reals\to\reals\) are Borel measurable maps and
\begin{equation*}
\underline{\phi}_{\epsilon}
\leqslant\phi
\leqslant\overline{\phi}_{\epsilon} 
\text{ and }\/\E^P\group[\big]{\overline{\phi}_{\epsilon}-\underline{\phi}_{\epsilon}}<\epsilon
\text{ for all \(\phi\in\Phi\)}.
\end{equation*}
If we can make sure that all\/ \(\Phi_\epsilon\), \(\epsilon>0\) are finite, then \(\Phi\) is a Glivenko--Cantelli class for~\(P\).
\end{theorem}
\noindent
% For a proof, we refer to Pollard~\cite{pollard2012convergence}. 
% Although Pollard gives the example of "k-means", also called quantisation, where this theorem can be applied, it is not that easy to check in practice.
Intuitively, \(\FF\), and as a consequence \(\Pee\), cannot be too large or contain too many `very different' maps, as this would tend to make the number of brackets infinite.
To formalise this intuition, we first define the \(L_1\)-seminorm \(\norm{\cdot}_{P,1}\) on the Borel measurable real-valued maps, corresponding to our probability measure~\(P\):
\begin{equation*}
\norm{\cdot}_{P,1}\coloneqq\E^{P}\group{\abs{\cdot}},
\end{equation*}
% Each equivalence class consist of all maps which are almost surely equal with respect to \(P\).
% A more in depth description can be found in~\cite{fristedt2013modern}.
% We will, with slight abuse of notation, write maps as arguments of the norms when we mean equivalence classes.
% We need these equivalence classes to ensure that \(\norm{\cdot}_{P,1}\) is indeed a norm, not merely a seminorm.
Then we introduce a so-called \emph{bracketing number} that measures how large the set \(\FF\) is.
Our definition adapts the one in~\cite[Definition 2.1.6]{van1996weak} to suit our present needs:

\begin{definition}\label{def:brack}
Let\/ \(\Phi\) be a subset of a seminormed space \((\mathscr{N},\norm{\cdot})\) of real maps.
Given two maps \(\ell\) and \(u\) in \(\mathscr{N}\), the bracket \([\ell,u]\) is the set of all maps \(\phi\in\Phi\) such that \(\ell\leqslant\phi\leqslant u\). 
An \(\epsilon\)-bracket is a bracket \([\ell,u]\) with \(\norm{u-\ell}<\epsilon\). 
The bracketing number~\(N_{[]}(\epsilon,\Phi,\norm{\cdot})\) is the smallest number of \(\epsilon\)-brackets needed to cover~\(\Phi\), meaning that \(\Phi\) is a subset of their union. 
\end{definition}
\noindent In this definition, the upper and lower bounds~\(u\) and~\(\ell\) of the brackets need not belong to~\(\Phi\) themselves, but do have finite seminorm as a consequence of \(\norm{u-\ell}<\epsilon\).

Using the bracketing number, we can restate \cref{th:brack} as the statement that a set~\(\FF\) is a (strong) Glivenko--Cantelli class for~\(P\) if
\begin{equation*}
N_{[]}(\epsilon,\FF,\norm{\cdot}_{P,1})<+\infty,\text{ for every \(\epsilon>0\).}
\end{equation*}
This is how it is also stated in \cite[Thm.~2.4.1]{van1996weak}.
If a class of maps can be covered by a finite collection of \(\epsilon\)-brackets---which are classes of maps not more than \(\epsilon\) apart---then it is a (strong) Glivenko--Cantelli class.
This bracketing number condition looks for a finite cover of the set~\(\FF\) of maps~\(f_t\) themselves, but we may wonder whether trying to find a finite cover for the parameter set~\(T\) itself does not make things easier.
This is what we now set out to discover.

We consider a number of special cases, beginning with the simplest of all: when \(T\) is finite.

\begin{proposition}\label{prop:finiT}
If the index set~\(T\) is finite, then \(\Eoh^{\Pee}(f)\) is a strongly consistent estimator.
\end{proposition}

Infinite sets~\(T\) are hard to deal with in general, but we can say a few useful things in a number of special cases.
To begin, we need a way to express that the index set~\(T\) can be covered by a finite number of `balls' of some maximum size.
In order to do this, we need to assume that we can measure distances in~\(T\), so \(T\) must be metrisable.
The following definition of a \emph{covering number} proves instrumental in measuring the size of a metrisable parameter set. 

\begin{definition}[{\protect\cite[Definition 2.2.3]{van1996weak}}]
Let \((T,d)\) be a metric space. 
Then the covering number \(N(\epsilon,T,d)\) is the smallest number of balls of radius \(\epsilon>0\) needed to cover~\(T\), meaning that \(T\) is a subset of the union of these balls.
\end{definition}

A useful theorem for connecting the bracketing number to the covering number relies on the Lipschitz continuity of the class of maps considered.
Our following result borrows the essential arguments in~\cite[Thm.~2.7.11]{van1996weak} to apply them to our present context.
It guarantees that a finite cover of the parameter set~\(T\) implies the existence of a finite cover for the class of maps~\(\FF\).

\begin{theorem}\label{th:lipschitz}
Let \(\FF\) be a subset of a seminormed space \((\mathscr{N},\norm{\cdot})\) of real maps.
Assume that
\begin{equation}\label{eq:lipschitz}
\abs{f_s(x)-f_t(x)}\leqslant d(s,t)F(x)\text{ for all \(x\in\reals\) and \(s,t\in T\),}
\end{equation}
for some metric \(d\) on the index set~\(T\), and some \(F\) in \(\mathscr{N}\).
Then \(N_{[]}(2\epsilon\norm{F},\FF,\norm{\cdot})\leqslant N(\epsilon,T,d)\).
\end{theorem}

For maps~\(f_T\) that are Lipschitz continuous in their parameter, as expressed by \cref{eq:lipschitz}, this reduces the problem of proving that a set of maps has finite bracketing number to the problem of proving that a set of parameters has finite covering number.

In practice, the parameter space~\(T\) will often be some part of a \emph{finite-dimensional vector space}, and then we can borrow the following simple result to further simplify the sufficient conditions for being a Glivenko--Cantelli class.

\begin{proposition}[{\protect\cite[Example 27.1]{shalev2014understanding}}]\label{prop:findim}
Consider a positive integer \(m\), a norm \(\norm{\cdot}\) on \(\reals^m\) and a bounded subset \(T \subset\reals^m\).
Let \(c\coloneqq\sup_{t\in T} \norm{t}\), then
\begin{equation*}
N(\epsilon,T,\norm{\cdot})\leqslant\left(\frac{2c\sqrt{m}}{\epsilon}\right)^m.
\end{equation*}
\end{proposition}
\noindent
\cref{th:consistent,th:brack,th:lipschitz} and Proposition~\ref{prop:findim} now combine into one straightforward theorem.

\begin{theorem}\label{th:combo}
Suppose that the set \(T \) that indexes~\(\Pee\) is a bounded subset of \(\reals^m\) and consider a norm \(\norm{\cdot}\) on \(\reals^m\).
If there is some Borel measurable map~\(F\colon\reals\to\reals\) with \(\norm{F}_{P,1}<+\infty\) for which we have the inequality
\begin{equation*}
\abs{f_s(x)-f_t(x)}\leqslant\norm{s-t}F(x)\text{ for all \(s,t\in T\) and \(x\in\reals\),}
\end{equation*}
then \(\Eoh_n^{\Pee}(f)\) is a strongly consistent estimator of\/~\(\Eo^{\Pee}(f)\).
\end{theorem}

Because even Lipschitz continuity can sometimes be hard to check, we may be able to replace it with a condition that is stronger, but more easily checked.
We will assume that the map \(f_T\) can be extended to a Borel measurable map~\(f_{T_c}\) on \(\reals\times T_c\), where \(T_c\) is some open set that includes the convex hull of~\(T\).
Furthermore, we will assume that for every \(t\in T\) it holds that \(f_t=f_{T_c}(\cdot,t)\) and therefore also \(\E^{P_t}(f(\cdot))=\E^P(f_{T_c}(\cdot,t))\).

\begin{theorem}\label{th:liptodiff}
Suppose that the set~\(T\) that indexes~\(\Pee\) is a bounded subset of\/~\(\reals^m\) for a norm~\(\norm{\cdot}\) on~\(\reals^m\), and that \(f_{T_c}\) is differentiable with respect to its second argument~\(t\) on all of~\(T_c\).
If there is some Borel measurable map~\(F\colon\reals\to\reals\) such that 
\begin{equation*}\label{eq:functions:for:importance:sampling2}
\norm{\nabla_t f_{T_c}(x,t)}\leqslant F(x)\text{ for all \(x\in\reals\) and \(t\in T_c\)},\footnote{The symbol \(\nabla_t\) denotes the gradient \(\group[\big]{\frac{\partial\,}{\partial t_1},\frac{\partial\,}{\partial t_2},...,\frac{\partial\,}{\partial t_m}}\) with respect to the argument vector \(t=(t_1,t_2,..,t_m)\in\reals^m\).}
\end{equation*} 
with \(\norm{F}_{P,1}<+\infty\), then \(\Eoh_n^{\Pee}(f)\) is a strongly consistent estimator of\/~\(\Eo^{\Pee}(f)\).
\end{theorem}
 
Finally, when the parameter set~\(T\) is a \emph{compact} subset of a finite-dimensional vector space, the sufficient conditions can be simplified even further.
Observe, by the way, that the supremum $\sup_{t\in T}\abs{f_t}$ below is measurable because a compact subset of a finite-dimensional space is separable, which allows us to reduce the supremum to a countable one, due to the continuity of $f_T$.
For the theorem below we assume that the map \(f_T\) can be extended to a Borel measurable map~\(f_{T_o}\) on \(\reals\times T_o\), where \(T_o\) is some open set that includes \(T\).
Furthermore, we assume that for every \(t\in T\) it holds that \(f_t=f_{T_o}(\cdot,t)\) and therefore also \(\E^{P_t}(f(\cdot))=\E^P(f_{T_o}(\cdot,t))\).

\begin{theorem}\label{th:compacT}
Suppose that the set \(T\) that indexes~\(\Pee\) is a compact subset of\/ \(\reals^m\) and \(T_o\) is an open set that includes~\(T\).
Assume that \(f_{T_o}\) is continuously differentiable in both arguments~\((x,t)\) on all of \(\reals\times T_o\), and that\/~\(\E^P\group[\big]{\sup_{t\in T}\abs*{f_t}}<+\infty\).
Then \(\Eoh^{\Pee}(f)\) is a strongly consistent estimator of\/~\(\Eo^{\Pee}(f)\).
\end{theorem}

\subsection{Consistency for Importance Sampling}\label{sec:consistency:for:importance:sampling}
In the special case of importance sampling, \cref{th:combo,th:liptodiff,th:compacT} can be further simplified.
In what follows we assume that all the assumptions made in \cref{sec:importance:sampling}, where we introduced importance sampling, are satisfied.

\begin{theorem}\label{th:comboimp}
Suppose that the set~\(T \) that indexes~\(\Pee\) is a bounded subset of\/~\(\reals^m\).
Consider any norm~\(\norm{\cdot}\) on~\(\reals^m\).
If there is some Borel measurable map~\(F\colon\reals\to\reals\) such that \(\int_\reals\abs{f(x)}F(x)\dx<+\infty\) and
\vspace{3pt}
\begin{equation*}
\abs{p_s(x)-p_t(x)}\leqslant\norm{s-t}F(x)\text{ for all \(s,t\in T\) and \(x\in\reals\),}\vspace{3pt}
\end{equation*}
then \(\Eoh_n^{\Pee/P}(f)\) is a strongly consistent estimator of\/~\(\Eo^{\Pee}(f)\).
\end{theorem}

For the following theorem, we assume that there is some open set~\(T_c\) that includes the convex hull of~\(T\) and some Borel measurable map \(p_{T_c}\colon\reals\times T_c\to\reals\) such that \(p_t=p_{T_c}(\cdot,t)\) for every \(t\in T\).

\begin{theorem}\label{th:liptodiffimp}
Suppose that the set~\(T\) that indexes~\(\Pee\) is a bounded subset of\/~\(\reals^m\), and assume that \(p_{T_c}\) is differentiable with respect to its second argument~\(t\).
Consider any norm~\(\norm{\cdot}\) on~\(\reals^m\).
If there is some Borel measurable map~\(F\colon\reals\to\reals\) such that \(\int_{\reals}\abs*{f(x)}F(x)\dx<+\infty \) and
\vspace{3pt}
\begin{equation*}
\norm{\nabla_t p_{T_c}(x,t)}\leqslant F(x)\text{ for all \(x\in\reals\) and \(t\in T_c\)},\vspace{3pt}
\end{equation*} 
then \(\Eoh_n^{\Pee/P}(f)\) is a strongly consistent estimator of\/~\(\Eo^{\Pee}(f)\).
\end{theorem}

Next, similarly to the previous section, we assume that there is some open set \(T_o\) that includes \(T\) and some Borel measurable map $p_{T_o}\colon\reals\times T_o\to\reals$ such that $p_{T_o}(x,t)\coloneqq p_t(x)$ for all $x\in\reals$ and $t\in T$.

\begin{theorem}\label{th:compacTimp}
Suppose that the set~\(T\) that indexes~\(\Pee\) is a compact subset of~\(\reals^m\), assume that \(p_{T_o}\) is continuously differentiable in both its arguments~\((x,t)\), that \(\int_{\reals}\sup_{t\in T}p_t(x)\dx<+\infty\), and that \(f\) is bounded.
Then \(\Eoh^{\Pee/P}(f)\) is a strongly consistent estimator of\/~\(\Eo^{\Pee}(f)\).
\end{theorem}

\section{Examples}\label{sec:examples}
\subsection{Consistency in a Practical Example}\label{sec:example:of:consistency}
As a first illustration, we prove the convergence of importance sampling for an example by Fetz~\cite{fetz2016imprecise} involving a beam bedded on a spring; see~\cref{fig:spring}.
The aim is to calculate the upper failure probability in the form~\(\overline{P}(g(X)\leqslant0)\) for a given map~\(g\colon\reals\to\reals\) and a normally distributed real random variable~\(X\) with parameters~\((\mu,\sigma)\in T\coloneqq[\muu,\muo]\times[\sigmu,\sigmo]\).
The map~\(g\) is given by
\begin{equation*}
g(x)\coloneqq M_{\text{yield}}-\frac{qL^2}{4}\max\vset[\bigg]{\frac{(1-c(x))^2}{2},c(x)-\frac12}, 
\end{equation*}
with 
\begin{equation*}
c(x)\coloneqq\frac{5x}{384\frac{EI}{L^3}+8x},
\end{equation*}
with beam length~\(L\), a uniformly distributed load~\(q\), elastic moment~\(M_{\text{yield}}\) and beam rigidity~\(EI\), see~\cite{fetz2016imprecise} for more details.

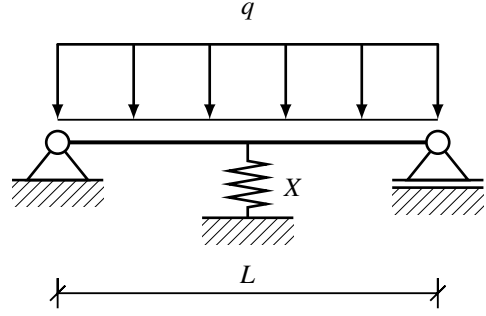
\begin{figure}[t] % bouwkundevoorjasper
\floatconts
  {fig:spring}
  {\caption{A beam bedded on a spring.}}
  {
    \begin{tikzpicture}
    \point{a}{0}{0};
    \point{b}{5}{0};
    \point{c}{2.5}{0};
    \beam{2}{a}{b};
    \support{1}{a};
    \support{2}{b};
    \support{5}{c};
    \hinge{1}{a};
    \hinge{1}{b};
    \lineload{1}{a}{b};
    \notation{1}{c}{\(q\)}[above =1.5]
    \notation{1}{c}{\(X\)}[below right=0.5]
    \dimensioning{1}{a}{b}{-2}[\(L\)];
    \end{tikzpicture}
  }
\end{figure}

In our context we can rewrite the failure probability as \(P(g(X)\leqslant 0)=1-\E^P\group{\mathbb{I}_{\vset{g(X)>0}}}\),\footnote{The notation \(\mathbb{I}_A\) represents the indicator of a set \(A\).} for which we want to find the maximal value over the available distributions.
This can be rephrased as finding the lower expectation
\begin{equation*}
\inf_{(\mu,\sigma)\in T}\E^{P_{(\mu,\sigma)}}\group[\big]{\mathbb{I}_{g(X)>0}},
\end{equation*}
and then subtract that from~$1$.
We want to do importance sampling to find this lower expectation, and take the central distribution to be normally distributed with parameters \((\mu_o,\sigma_o)\) with \(\mu_o\in\reals\) and \(0<\sigma_o\).
With \(P_o\) we denote the probability measure of this distribution.
Observe that this pair of parameters \((\mu_o,\sigma_o)\) need not be contained in the parameter set \([\muu,\muo]\times[\sigmu,\sigmo]\).
For importance sampling, we define our set of Borel measurable maps as \(\FF=\cset{f_{(\mu,\sigma)}}{(\mu,\sigma)\in T}\) where, by \cref{eq:functions:for:importance:sampling}:
\begin{multline*}
f_{(\mu,\sigma)}\colon\reals\to\reals\colon\\ 
x\mapsto
\mathbb{I}_{\cset{y}{g(y)>0}}(x)
\frac{\sigma_o}{\sigma}
\exp\group[\Big]{-\frac{(x-\mu)^2}{2\sigma^2}+\frac{(x-\mu_o)^2}{2\sigma_o^2}}.
\end{multline*}
To prove consistency, we will show how multiple theorems can be used in decreasing complexity.
We will start with \cref{th:liptodiff}.
We are lucky, since \(T\) is already convex.
The map~\(f_{(\mu,\sigma)}\) is measurable in its argument~\(x\) since it is a product of the indicator of a level set of a measurable map and another measurable map.
Furthermore, it is clearly differentiable in the parameters~\(\mu\) and~\(\sigma>0\).
We verify in the Appendix \iftoggle{extendedversion}{}{of the arXiv version~\cite{extpaper}} that a map that can serve as the~\(F\) in \cref{th:liptodiff} is given by 
\begin{multline*}
F(x)\\
=\frac{\sigma_o}{\sigmu^2}
\begin{dcases}
\group[\Big]{\frac{(x-\muo)^2}{\sigmu^2}+1}
e^{\frac{(x-\mu_o)^2}{2\sigma_o^2}-\frac{(x-\muu)^2}{2\sigmo^2}}
&x<\muu\\
\group[\Big]{\frac{(\muo-\muu)^2}{\sigmu^2}+1}
e^{\frac{(x-\mu_o)^2}{2\sigma_o^2}}
&\muu\leqslant x<\muo\\
\group[\Big]{\frac{(x-\muu)^2}{\sigmu^2}+1}
e^{\frac{(x-\mu_o)^2}{2\sigma_o^2}-\frac{(x-\muo)^2}{2\sigmo^2}}
&\muo\leqslant x.
\end{dcases}\\[-5pt]
\end{multline*}
We can conclude that the estimator is consistent.

Alternatively, we can use \cref{th:liptodiffimp}.
We are again lucky, since \(T\) is still convex.
The map~\(p_{(\mu,\sigma)}\) is also measurable in its argument~\(x\) and clearly differentiable in the parameters~\(\mu\) and~\(\sigma>0\).
Using a similar argument as above, we verify in the Appendix \iftoggle{extendedversion}{}{of~\cite{extpaper}} that the following map can serve as the \(F\) in \cref{th:liptodiffimp}:
\begin{multline*}
F(x)\\
=\frac{1}{\sqrt{2\pi}\sigmu^2}
\begin{dcases}
\group[\Big]{\frac{(x-\muo)^2}{\sigmu^2}+1}
e^{-\frac{(x-\muu)^2}{2\sigmo^2}}
&x<\muu\\
\frac{(\muo-\muu)^2}{\sigmu^2}+1
&\muu\leqslant x<\muo\\
\group[\Big]{\frac{(x-\muu)^2}{\sigmu^2}+1}
e^{-\frac{(x-\muo)^2}{2\sigmo^2}}
&\muo\leqslant x.
\end{dcases}\\[-5pt]
\end{multline*}

Finally we will use \cref{th:compacTimp}.
The set \(T\) is a compact subset of \(\reals^2\) because it is closed and bounded and since \(f\) is an indicator, it is a bounded map.
Furthermore, \(p_{(\mu,\sigma)}\) is clearly continuously differentiable in the parameters~\(\mu\) and~\(\sigma>0\).
We show in the Appendix \iftoggle{extendedversion}{}{of~\cite{extpaper}} that
\begin{equation*}
\int_{\reals} \sup_{t\in T} p_t(x)\dx\leqslant \frac{1}{\sigmu}\group*{\sigmo+\frac1{\sqrt{2\pi}}(\muo-\muu)}<+\infty,
\end{equation*}
from which we can conclude consistency of the estimator by \cref{th:compacTimp}.

\subsection{An Example of No Consistency}\label{sec:example:of:no:consistency}
Next, we look at a theoretical example where the importance sampling estimator does not converge---shows no consistency.
We consider a map~\(f\) that is greater than one everywhere, and a central uniform distribution \(P\) on \([0,2]\), with density \(p\coloneqq\frac12 \mathbb{I}_{[0,2]}\).
We will construct a countably infinite set of distributions for which the technique of importance sampling does not work.
First, we define the set \(\text{Bi}\) of all finite binary sequences of even length with the same number of zeroes and ones.
Using this set, we define the countable set of densities
\begin{equation*}
D
\coloneqq\cset[\bigg]{\sum_{\ell=1}^{2k}a_{\ell}\mathbb{I}_{\left[\frac{\ell-1}{k},\frac{\ell}{k}\right]}}
{k\in\naturals\text{ and }(a_1,\dots,a_{2k})\in\text{Bi}},
\end{equation*}
and we denote the corresponding set of probability measures by~\(\Pee_D\).
In \cref{fig:densities}, two examples of densities in \(D\) are plotted.
The importance sampling method will not work in this example, because for every finite observed sample \(x_1,x_2,x_3,...,x_n\) there is a binary sequence of size \(2n\) that corresponds to a density that is zero on all the sampled values.
Consequently, the estimate \(\Eoh^{\Pee_D}(f)\) will be identically zero, yet the real lower expectation must be positive since the map~\(f\) is.

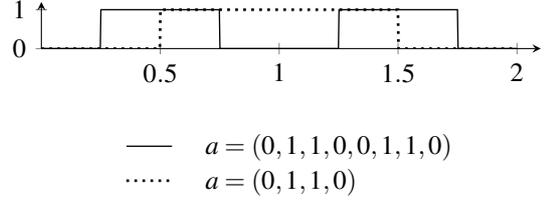
\begin{figure}[t]
\floatconts
  {fig:densities}
  {\caption{Some examples of densities in \(D\).}}
  {
    \begin{tikzpicture}
    \begin{axis}[
      height=2.2cm,
      width=\columnwidth,
      axis y line=left,
      axis x line=middle,
      ytick={0,1},
      ymax=1.2,
      xmax=2.1,
     % title=Spread of the Virus,  
     % xlabel=x,
     % ylabel=density value
     legend style={draw=none,at={(0.5,-1.6)},anchor=north}
     ]

    \addplot[mark=none, domain=0:2,samples=800,style={line width=.5pt}]{ (x>0.25 && x<=0.75) || (x>1.25 && x<1.75) };
    %\addplot[dashed,mark=none, domain=0:2,samples=800,style={{line width=1.5pt}}]{ x<1 };
    \addplot[dotted,mark=none, domain=0:2,samples=800,style={{line width=1pt}}]{ (x>0.5 && x<1.5)  };

    % Legend
    \addlegendentry{\quad\(a=(0,1,1,0,0,1,1,0)\)}
    \addlegendentry{\quad\(a=(0,1,1,0)\phantom{,0,1,1,0}\)}

    \end{axis}
    \end{tikzpicture}
  }
\end{figure}

However, the importance sampling estimator will be consistent for lower expectations associated with finite subsets of \(D\), by Proposition~\ref{prop:finiT}.

So we find that, in contrast with finite sets of maps~\(\FF\), countable sets can break the consistency of the importance sampling estimator.
The reason that this happens is that there are only a finite number of samples.
If the set~\(\FF\) of maps~\(f_t\) is countably infinite, then it is possible that some differences between two such maps remain `unexplored' by finite sampling.

In summary then, the reason why consistency breaks in our example is because the sample is finite. 
In particular, the set~\(D\), or equivalently, the set~\(\text{Bi}\) is then large enough to make sure that there is always at least one density~\(p_a\) for some~\(a\in\text{Bi}\) that is zero in the sample.
But when we restrict ourselves to some finite subset of~\(D\), and if the sample is sufficiently large, then this finite set of densities will no longer be large enough to make sure that it will have at least one member that is zero in the sample.

\section{Conclusion}\label{sec:conclusion}
We have studied Monte Carlo estimators in the context of imprecise probabilities.
We have given a practical general form of an estimator of the lower expectation and two explicit constructions that lead to an estimator of this form: inverse transform sampling and importance sampling.
For this general form, we have investigated bias and consistency.
Under fairly non-restrictive assumptions, the estimator bias was proved to be negative and non-decreasing, or in other words, conservative and shrinking in absolute value with increasing sample size.
Consistency, on the other hand, cannot be proved in general: as we have shown in one of our examples, the estimator can remain inconsistent even for a countable set of distributions.
But we have investigated tools that can be used to prove consistency of our estimator.
For instance, we showed that the case where the `distributions' are Lipschitz continuous over a bounded set of parameters leads to consistent estimators.
For compact parameter sets, we further simplified this continuity condition.
In our first example, we showed how consistency can be proven for the importance sampling estimator used in~\cite{fetz2016imprecise}, in multiple ways.
The second example showed how consistency can fail even in the case of a countable set of of parameters.

Future work will deal with other limit laws for the Monte Carlo estimator~\eqref{eq:standardform} for lower expectations, and extension of these methods to an imprecise version of the Markov Chain Monte Carlo estimator.

\acks{
We would like to thank Matthias Troffaes for many stimulating discussions on the topic of Monte Carlo simulation for imprecise probabilities.
Furthermore, we thank the reviewers for their extensive and thoughtful comments.
}

\bibliography{isipta2019-template}
% should be alandvonvereweg19.bib if this were a submission

\iftoggle{extendedversion}{
\onecolumn
\clearpage
\appendix

\section{Proofs of Various Results in the Paper}
\subsection{Inverse Transform Sampling}%\label{sec:inverse:transform:sampling}

\begin{lemma}\label{lem:basisf}
For any cdf \(F\) and its pseudo-inverse \(F^{\dagger}\) defined as in \cref{cumin}, the following statements are true: 
\begin{enumerate}[label=\upshape(\roman*),leftmargin=*]
\item\label{lem:minac} In~\cref{cumin}, if \(x\in (0,1)\) then the infimum is real and achieved; 
\item\label{stat1} \((\forall x\in\reals) F^{\dagger}(F(x))\leqslant x\);
\item\label{stat2} \((\forall x\in (0,1)) F(F^{\dagger}(x))\geqslant x\);
\item\label{stat3} \((\forall x\in\reals)(\forall u\in(0,1)) u\leqslant F(x) \ifandonlyif F^{\dagger}(u) \leqslant x\).
\end{enumerate}
\end{lemma}

\begin{proof}
% Checked and improved by Gert
For~\ref{lem:minac} and~\ref{stat2}, fix \(x\in(0,1)\) and let \(A(x)\coloneqq\cset{y\in\reals}{x \leqslant F(y)}\).
We will first prove that \(\inf A(x) \in\reals\) by contradiction, by deriving a contradiction for the two other cases: \(\inf A(x)=+\infty\) and \(\inf A(x)=-\infty\).

% We look at three cases.
As a first case, suppose that \(\inf A(x)=+\infty\).
Since \(A(x)\) is a set of real numbers, this is only true if \(A(x)\) is the empty set.
% The first one is obtained when \(A(x)\) is empty.
% Then \(\inf A(x)=+\infty\), which is not a real number, so \ref{lem:minac} holds trivially.
% We now show that \(A(x)\) can only be empty if \(x=1\).
Since \(x\in(0,1)\), we know that we can we can write \(x=1-\epsilon\) for some \(\epsilon>0\).
From the property of cdf's that
\(
\lim_{y\to +\infty} F(y)=1,
\)
it follows, by the definition of limits, that there is some real number \(R\) such that \(1-F(y)<\epsilon\) for all \(y>R\).
So \(x=1-\epsilon<F(y)\) for all \(y>R\), so all \(y>R\) are included in the set \(A(x)\), which means that it cannot be empty, a contradiction.

As a second case, suppose that \(\inf A(x)=-\infty\).
Since \(x\in(0,1)\), we know that \(x>0\).
% For this case \(x=1\), \ref{stat2} is not stated.
% The second case is that \(A(x)=\reals\), where \(\inf A(x)=-\infty\), which is also not a real number, so \ref{lem:minac} also holds trivially.
% As we will show, this will always happen when \(x=0\) and never for other values of \(x\) for similar reasons as the previous case.
% Suppose that \(x=\epsilon\) for some \(0<\epsilon\leqslant1\).
From the property of cdf's that
\(
\lim_{y\to -\infty} F(y)=0,
\)
it follows, by the definition of limits, that there is some real number \(R\) such that if \(y<R\) then \(F(y)<x\), or equivalently, if \(F(y)\geq x\) then $y\geq R$.
This tells us that \(R\) is a real lower bound of the set \(A(x)\), whence \(\inf A(x)\geq R\), a contradiction.

Since \(F\) is non-decreasing, \(A(x)\) is an up-set, meaning that \(y_1\in A(x)\) and \(y_1 \leqslant y_2\) implies that \(y_2\in A(x)\).
Hence, since \(A(x)\) is a proper non-empty subset of \(\reals\), it is of the form \([F^{\dagger}(x),+\infty)\) or \((F^{\dagger}(x),+\infty)\), where \(F^\dagger(x)=\inf A(x)\) is real, as we have just shown.
If follows from the definition of an infimum that there is a non-increasing sequence \(y_n\) of elements of \(A(x)\) that converges to \(F^{\dagger}(x)\), so \(y_n \downarrow F^{\dagger}(x)\) and therefore also \(F(y_n)\downarrow F(F^{\dagger}(x))\), since \(F^\dagger(x)\) is real and \(F\) is right-continuous and non-decreasing. 
Hence, \(F(F^{\dagger}(x))=\lim_{n\to+\infty} F(y_n)\geqslant x\), and therefore \(F^{\dagger}(x)\in A(x)\) proving~\ref{stat2}, and also~\ref{lem:minac}.

For~\ref{stat1}, the definition of \(F^{\dagger}\) implies that, indeed, \(F^{\dagger}(F(x))=\inf \cset{ y\in\reals}{F(x)\leqslant F(y)}\leqslant x\).

Finally, for~\ref{stat3}, observe that
\begin{equation*}
u \leqslant F(x) 
\ifandonlyif x\in A(u) 
\ifandonlyif F^{\dagger}(u) \leqslant x,
\end{equation*}
where the first equivalence follows from the definition of \(A(u)\) above, and the second from~\ref{lem:minac} --- which guarantees that \(F^\dagger(u)=\min A(u)\) --- and the fact that \(A(u)\) is an up-set.
\end{proof}

\begin{theorem}
Consider a cdf \(F\colon\reals\to\reals\) and a uniformly distributed real random variable \(U\)on the unit interval \((0,1)\). 
If we define \(F^{\dagger}\) as in \cref{cumin}, then the real random variable \(X\coloneqq F^{\dagger}(U)\) has cdf \(F\). 
\end{theorem}

\begin{proof}
% Checked, corrected and improved by Gert
The cdf of \(X\) is given by, for any $x\in\reals$:
\begin{align*}
F_X(x)
&\coloneqq P(X\leqslant x)=P(F^{\dagger}(U)\leqslant x)\\
&=\lambda\group{\cset{u\in(0,1)}{F^{\dagger}(u)\leqslant x}}
=\lambda\group{\cset{u\in(0,1)}{u\leqslant F(x)}}
=\lambda\group{(0,F(x)]}\\
&=F(x),
\end{align*}
where \(\lambda\) is the Lebesgue measure relative to \((0,1)\), and the fourth equality follows from Lemma~\ref{lem:basisf}\ref{stat3}.
\end{proof}

\subsection{Estimator Properties}
In the main body of the paper we only needed the outer cover \(f^*\) of a map \(f\colon \reals^\infty\to\reals\). 
For the following results however, we will also need inner covers, which we define as \(f_*\coloneqq-(-f)^*\) in the corresponding probability space.

\begin{lemma}[Super-additivity of inner expectations] \label{lem:subad}
Consider two maps \(f,g\colon\reals^{\infty}\to\reals\cup\vset{-\infty}\) for which\/ \(\Eo^\infty(f)+\Eo^\infty(g)\) is well-defined, meaning that \(\Eo^\infty(f)<+\infty\) or \(\Eo^\infty(g)>-\infty\), and \(\Eo^\infty(f)>-\infty\) or \(\Eo^\infty(g)<+\infty\).
Then\/ \(f+g\) is well-defined and assumes values in \(\reals\cup\vset{-\infty}\), and\/ \(\Eo^\infty(f)+\Eo^\infty (g)\leqslant\Eo^\infty(f+g)\).
\end{lemma}
\begin{proof}

% Checked and improved by Gert
The first statement is immediate, so we turn to the proof of the inequality.
We consider two cases: for the first case we assume that both \(\Eo^\infty(f)>-\infty\) and  \(\Eo^\infty(g)>-\infty\) and for the second case we assume, without loss of generality, that \(\Eo^\infty(f)=-\infty\).

For the first case, we infer from the assumption that \(\Eo^\infty(f)>-\infty\) and \(\Eo^\infty(g)>-\infty\), and the relation between the inner and the outer expectation, that \(\Eu^\infty(-f)<+\infty\) and \(\Eu^\infty(-g)<+\infty\). 
Lemma~\ref{lem:mescov} then tells us that \((-f)\) and \((-g)\) have respective outer measurable covers \((-f)^*\) and \((-g)^*\), and that \(\Eu^\infty(-f)=\E^{\infty}((-f)^*)\) and \(\Eu^\infty(-g)=\E^{\infty}((-g)^*)\).

We will now prove that \(-f-g\) has a measurable cover \((-f-g)^*\) for which \(\Eu^\infty(-f-g)=\E^\infty((-f-g)^*)\).
By Lemma~\ref{lem:mescov} suffices to show that \(\Eu^\infty(-f-g)<+\infty\), where 
\begin{equation*}
\Eu^\infty(-f-g)
=\inf\cset[\big]{\E^{\infty}(h)}
{-f-g\leqslant h\in \extendedreals^{\reals^\infty}\text{ measurable and }\E^{\infty}(h)\text{ exists}}.
\end{equation*}
\((-f)^*+(-g)^*\) is measurable and, by definition \(-f\leqslant(-f)^*\) and \(-g\leqslant(-g)^*\), so \(-f-g\leqslant(-f)^*+(-g)^*\).
Furthermore, since \(\E^\infty((-f)^*)=\Eu^\infty(-f)<+\infty\) and \(\E^\infty((-g)^*)=\Eu^\infty(-g)<+\infty\), the expectation \(\E^\infty(((-f)^*+(-g)^*)^+)\leqslant\E^\infty(((-f)^*)^+) + \E^\infty(((-g)^*)^+)<+\infty\),\footnote{\dots since for all \(r,s \in\reals^{\reals^\infty}\), \((r+s)^+=\max\{r+s,0\}\leqslant\max\{r,0\}+\max\{s,0\}=r^+ + s^+\)} so \(\E^\infty((-f)^*+(-g)^*)\) exists.
We can conclude that
\begin{equation*}
\E^\infty((-f)^*+(-g)^*) \in \cset{\E^{\infty}(h)}
{-f-g\leqslant h\in \extendedreals^{\reals^\infty}\text{ measurable and }\E^{\infty}(h)\text{ exists}},
\end{equation*}
and therefore
\begin{equation*}
\Eu^\infty(-f-g)\leqslant\E^\infty((-f)^*+(-g)^*)=\E^\infty((-f)^*)+\E^\infty((-g)^*)<+\infty,
\end{equation*}
as needed.
So, indeed, \(-f-g\) has a measurable cover \((-f-g)^*\), for which \(\Eu^\infty(-f-g)=\E^\infty((-f-g)^*)\).

We can now apply Lemma~1.2.2(i) in \cite{van1996weak} to \((-f)\) and \((-g)\) to find that
\((-f)^*+(-g)^*\geqslant(-f-g)^*\) a.s., and therefore also \(-(-f)^*-(-g)^*\leqslant-(-f-g)^*\) a.s.,
so after taking the expectation on both sides, we get
\begin{equation*}
\E^\infty(-(-f)^*-(-g)^*)\leqslant\E^\infty(-(-f-g)^*)
\end{equation*}
and therefore also
\begin{equation*}
\E^\infty((-f-g)^*)\leqslant\E^\infty((-f)^*+(-g)^*)=\E^\infty((-f)^*)+\E^\infty((-g)^*),
\end{equation*}
where the equality, as before, follows from the linearity of expectations, and \(\E^\infty((-f)^*)=\Eu^\infty(-f)<+\infty\) and \(\E^\infty((-g)^*)=\Eu^\infty(-g)<+\infty\).
Since, by definition, \(\Eo^{\infty}(f)=-\Eu^{\infty}(-f)=-\E^{\infty}\group{(-f)^*}\), and similarly for the other expectations, this indeed leads to the desired inequality \(\Eo^\infty(f)+\Eo^\infty(g)\leqslant\Eo^\infty(f+g)\).

For the second case, we infer from the assumption that \(\Eo^\infty(f)=-\infty\) and that \(\Eo^\infty(f)+\Eo^\infty(g)\) is well-defined, that \(\Eo^\infty (g)<+\infty\).
Hence, \(\Eo^\infty(f)+\Eo^\infty (g)=-\infty\), and the desired inequality is trivially true.
\end{proof}

\subsubsection{Proof of Theorem~\ref{th:decrimp}} %%%%%%%%%%%%%%%%%%%%%%%%%%%%%%%%%%%%%%%%%%%%%% NOK

\begin{proofof}{\cref{th:decrimp}}
% Checked, corrected and improved by Gert
Later in this proof we will want to calculate inner and outer expectations of, and apply Lemma~\ref{lem:subad} to, combinations of the maps \(\Eoh_{n}^{\Pee}(f)\) for any positive integer \(n\)---for the sake of simplicity, we will use the notation \(g_n=g(X_{1:n}^P)\coloneqq\Eoh_{n}^{\Pee}(f)\).
In order to be allowed to do this, we want to show here that \(-\infty<\Eo^{\infty}(g_n)\), because this guarantees the existence of a measurable inner cover \(\group{g_n}_*\) such that \(\Eo^{\infty}(g_n)=\E^{\infty}(\group{g_n}_*)\), by Lemma~\ref{lem:mescov}.
We give a proof by induction.
First of all, it follows from our assumption that \(-\infty<\Eo^{\infty}(\inf_{t\in T}f_t)=\Eo^{\infty}(g_1)\), that the statement holds for \(n=1\).
Next, we assume that the statement holds for \(n=k\), and we show that it then also holds for \(n=k+1\).
So suppose that \(-\infty<\Eo^{\infty}\group{g_k}\), then we are guaranteed that \(g_k\) has a measurable inner cover \(\group{g_k}_*\), and that \(\Eo^{\infty}(g_k)=\E^{\infty}(\group{g_k}_*)\).
Similarly, by the construction of the underlying probability space, it follows that \(-\infty<\Eo^{\infty}\group{g_1(X_{1}^P)}=\Eo^{\infty}\group{g_1(X_{k+1}^P)}\) and this guarantees that \(g_1(X_{k+1}^P)\) has a measurable inner cover \(\group{g_1}_*\), and that \(\Eo^{\infty}(g_1(X_{k+1}^P))=\E^{\infty}(\group{g_1}_*)\).
Then \(h_{k+1}\coloneqq\frac1{k+1}\group{g_1}_*+\frac{k}{k+1}\group{g_k}_*\) is a measurable map, and
\begin{equation*}
\E^{\infty}\group{h_{k+1}}
=\E^{\infty}\group[\bigg]{\frac1{k+1}\group{g_1}_*+\frac{k}{k+1}\group{g_k}_*}
\geqslant\frac1{k+1}\E^{\infty}\group{\group{g_1}_*}+\frac{k}{k+1}\E^{\infty}\group{\group{g_k}_*}
>-\infty.
\end{equation*}
Moreover, it follows from the definition of a measurable inner cover that
\begin{align*}
h_{k+1}
\leqslant\frac1{k+1}\inf_{t\in T}f_t(X_{k+1}^P)+\frac{k}{k+1}\Eoh_{k}^{\Pee}(f)
&=\frac1{k+1}\inf_{t\in T}f_t(X_{k+1}^P)+\frac{k}{k+1}\inf_{t\in T}\frac1k\sum_{\ell=1}^{k}f_t(X_{\ell}^P)\\
&\leqslant\inf_{t\in T}\frac1{k+1}\group[\bigg]{f_t(X_{k+1}^P)+\sum_{\ell=1}^{k}f_t(X_{\ell}^P)}
=g_{k+1}(X_{1:k+1}^P),
\end{align*}
which implies that, indeed, \(\Eo^{\infty}\group{g_{k+1}}\geqslant\E^{\infty}\group{h_{k+1}}>-\infty\), and that, therefore, \(g_{k+1}\) also has a measurable inner cover.

We are now ready for the proof of the theorem.
First, we prove that 
\begin{equation*}
\Eo^{\infty}\group[\Big]{\Eoh_n^{\Pee}(f)}
\geqslant
\Eo^{\infty}\group[\Big]{\Eoh_{n-1}^{\Pee}(f)}.
\end{equation*}
We rewrite
\begin{equation*}
\Eoh_n^{\Pee}(f)
=\inf_{t\in T}\frac1n\sum_{k=1}^{n}f_t(X^P_k)
=\inf_{t\in T}\frac1n\sum_{k=1}^{n}\frac1{n-1}\sum_{\substack{j=1\\j\neq k}}^{n}f_t(X^P_j)
\geqslant\frac1n\sum_{k=1}^{n}\inf_{t\in T}\frac1{n-1}\sum_{\substack{j=1\\j\neq k}}^{n}f_t(X^P_j).
\end{equation*}
Now we take the lower expectation \(\Eo^{\infty}\) of both sides of the inequality to get\footnote{\label{fn:monotonicity}It is clear from their definitions near the end of Section~\ref{sec:imprecise:monte:carlo} that both the upper and lower expectations are monotone.}
\begin{align*}
\Eo^{\infty}\group[\Big]{\Eoh_n^{\Pee}(f)}
&\geqslant\Eo^{\infty}\group[\bigg]{\frac1n\sum_{k=1}^{n}\inf_{t\in T}\frac1{n-1}\sum_{\substack{j=1\\j\neq k}}^{n}f_t(X^P_j)}
\geqslant\frac1n\sum_{k=1}^{n}\Eo^{\infty}\group[\bigg]{\inf_{t\in T}\frac1{n-1}\sum_{\substack{j=1\\j\neq k}}^{n}f_t(X^P_j)}\\
&=\frac1n\sum_{k=1}^{n}\Eo^{\infty}\group[\Big]{\Eoh_{n-1}^{\Pee}(f)}
=\Eo^{\infty}\group[\Big]{\Eoh_{n-1}^{\Pee}(f)},
\end{align*}
where the second inequality follows from Lemma~\ref{lem:subad} and the first equality from the fact that the underlying probability space is an independent product space.

Secondly, we want to prove that
\begin{equation*}
\Eo^{\infty}\group[\Big]{\Eoh_n^{\Pee}(f)}
\leqslant 
\Eu^{\infty}\group[\Big]{\Eoh_n^{\Pee}(f)}.
\end{equation*}
We distinguish between two cases.
The first case is where \(\Eu^{\infty}\group[\big]{\Eoh_n^{\Pee}(f)}=+\infty\) and then the inequality holds trivially.
In the second case, where \(\Eu^{\infty}\group[\big]{\Eoh_n^{\Pee}(f)}<+\infty\), we are guaranteed the existence of a measurable outer cover \((g_n)^*\) for which \(\Eu^\infty(g_n)=\E^\infty((g_n)^*)\).
Since \(\Eoh_n^{\Pee} (f)\) was proved above to also have a measurable inner cover \((g_n)_*\), we can use their definitions to find
\begin{equation*}
(g_n)_*\leqslant g_n\leqslant (g_n)^*,
\end{equation*}
and by taking the expectation on both measurable sides, we find that, indeed,
\begin{equation*}
\Eo^{\infty}\group[\Big]{\Eoh_n^{\Pee}(f)}
=\E^\infty((g_n)_*)
\leqslant\E^\infty((g_n)^*)
=\Eu^{\infty}\group[\Big]{\Eoh_n^{\Pee}(f)}.
\end{equation*}

Thirdly, we prove that
\begin{equation*}
\Eu^{\infty}\group[\Big]{\Eoh_n^{\Pee}(f)}
\leqslant 
\Eo^{\Pee}(f).
\end{equation*}
We start again from the left-hand side:
\begin{align*}
\Eu^{\infty}\group[\Big]{\Eoh_n^{\Pee}(f)}
&=\Eu^{\infty}\group[\bigg]{\inf_{t\in T}\frac1n\sum_{k=1}^nf_t(X_k^P)}\\
% &=\inf\cset{ \E (g)}{g \geqslant\inf_{t\in T} \frac1n \sum_{k=1}^n f_t(X_k), g\colon\reals\to\extendedreals \text{ measurable and \(\E^{\infty}(g)\) exists}}\\
% &\leqslant\inf\cset{\inf_{t\in T} \E (g_t)}{g_t \geqslant \frac1n \sum_{k=1}^n f_t(X_k), g_t\colon\reals\to\extendedreals \text{ measurable and \(\E(g_t)\) exists}}\\
% &=\inf_{t\in T}\inf\cset{ \E (g_t)}{g_t \geqslant \frac1n \sum_{k=1}^n f_t(X_k), g_t\colon\reals\to\extendedreals \text{ measurable and \(\E(g_t)\) exists}}\\
&\leqslant\inf_{t\in T}\Eu^{\infty}\group[\bigg]{\frac1n\sum_{k=1}^nf_t(X_k^P)}
=\inf_{t\in T}\E^\infty\group[\bigg]{\frac1n\sum_{k=1}^nf_t(X_k^P)}
=\inf_{t\in T}\E^P\group{f_t}
=\inf_{t\in T}\E^{P_t}\group{f}
=\Eo^{\Pee}(f),
% &\leqslant\inf_{t\in T}  \Eo^{\infty}\left(\frac1n \sum_{k=1}^n f_t(X_k^P)\right)\\
% &\leqslant\inf_{t\in T}  \E^{P}\left(\frac1n \sum_{k=1}^n f_t(X_k^P)\right)\\ 
% % of = bij deze laatste
% &=\inf_{t\in T} \frac1n \sum_{k=1}^n \E^{P}\left( f_t(X_k^P)\right)\\
% &=\inf_{t\in T} \E^{P_t} \left( f \right)\\
% &=\Eo^{\Pee} (f).
\end{align*}
where the first inequality follows from the monotonicity of the upper expectation.\footnote{See also footnote~\ref{fn:monotonicity}.}
\end{proofof}

\subsubsection{Proof of Theorem~\ref{th:consistent}}
First, we prove an auxiliary lemma.

\begin{lemma}\label{lem:supinf}
Consider a non-empty set \(A\) and two maps \(\phi,\psi\colon A \to\reals\), at least one of which is bounded below. 
Then
\begin{equation*}
\abs[\bigg]{\inf_{a\in A}\phi(a)-\inf_{a\in A}\psi(a)} 
\leqslant\sup_{a\in A}\abs*{\phi(a)-\psi(a)}.
\end{equation*}
\end{lemma}

\begin{proof}
% Checked and improved by Gert
We will look at two cases.
The first where one of the infima \(\inf_{a\in A}\phi(a)\) and \(\inf_{a\in A}\psi(a)\) is \(-\infty\) and the other where none of them are.
In the first case, we may assume without loss of generality \(\inf_{a\in A}\phi(a)=-\infty\), and so it follows from the assumptions that \(-\infty<\inf_{a\in A}\psi(a)\).
It follows from \(\inf_{a\in A}\phi(a)=-\infty\) that there is some sequence \(a_n\) in \(A\) such that \(\phi(a_n)\downarrow-\infty\). 
Hence, there is some positive integer \(N\) such that for every integer \(n\geq N\) it holds that \(\phi(a_n)<\inf_{a\in A}\psi(a)\).
If we consider the sequence \(b_n\coloneqq a_{N+n}\), then it follows that \(\lim_{n\to\infty} \abs*{\phi(b_n)-\psi(b_n)}=\lim_{n\to\infty} (\psi(b_n)-\phi(b_n))=+\infty\). 
This implies that \(\sup_{a\in A}\abs*{\phi(a)-\psi(a)}=+\infty\), and since the left-hand side of the desired inequality is well-defined --- an  extended real number --- because of the theorem's assumptions, the desired inequality holds trivially.

In the second case, both \/ \(-\infty<\inf_{a\in A}\phi(a)\) and \(-\infty<\inf_{a\in A}\psi(a)\), so both infima are real numbers.
Then
\begin{align*}
\abs[\bigg]{\inf_{a\in A}\phi(a)-\inf_{a\in A}\psi(a)}
&=\max\vset[\bigg]{\inf_{a\in A}\phi(a)-\inf_{a\in A}\psi(a),\inf_{a\in A}\psi(a)-\inf_{a\in A}\phi(a)}\\
&=\max\vset[\bigg]{-\inf_{a\in A}\group[\Big]{\psi(a)-\inf_{b\in A}\phi(b)},-\inf_{a\in A}\group[\Big]{\phi(a)-\inf_{b\in A}\psi(b)}}\\
&=\max\vset[\bigg]{\sup_{a\in A}\group[\Big]{\inf_{b\in A}\phi(b)-\psi(a)},\sup_{a\in A}\group[\Big]{\inf_{b\in A}\psi(b)-\phi(a)}}\\
&=\sup_{a\in A}\max\vset[\bigg]{\inf_{b\in A}\phi(b)-\psi(a),\inf_{b\in A}\psi(b)-\phi(a)}\\
&\leqslant\sup_{a\in A}\max\vset{\phi(a)-\psi(a),\psi(a)-\phi(a)}
=\sup_{a\in A}\abs*{\phi(a)-\psi(a)},
\end{align*}
which completes the proof.
\end{proof}

We are now ready to prove \cref{th:consistent}.

\begin{proofof}{\cref{th:consistent}}
% Checked, corrected and improved by Gert
That \(\FF\) is a (strong) Glivenko--Cantelli class for \(P\) means that 
\begin{equation} \label{eq:supcon}
P^{\infty}\left(\lim_{n\to +\infty}\left(\sup_{t\in T} \abs*{ \hat{\E}^{P}_n(f_t)-\E^{P}(f_t) }\right)^*=0 \right)=1.
\end{equation}
Fix any $n\in\naturals$.
By the $\Pee$-integrability assumption for $f$ and the definition of the~\(f_t\), we have that
\begin{equation*}
-\infty
<-\Eu^{\Pee}(\abs{f})
=\Eo^{\Pee}(-\abs{f})
\leq\Eo^{\Pee}(f)
=\inf_{t\in T}\E^{P_t}(f)=\inf_{t\in T}\E^P(f_t).
\end{equation*}
Hence, we can apply Lemma~\ref{lem:supinf}, and find
\begin{equation*}
\sup_{t\in T}\abs*{\hat{\E}^{P}_n(f_t)-\E^{P}(f_t)}
\geqslant
\abs*{\inf_{t\in T}\hat{\E}^{P}_n(f_t)-\inf_{t\in T}\E^{P}(f_t)}
=\abs*{\Eoh^{\Pee}_n(f)-\inf_{t\in T}\E^{P_t}(f)}
=\abs*{\Eoh^{\Pee}_n(f)-\Eo^{\Pee}(f)}
\geqslant0,
\end{equation*}
where the first equality follows from \cref{eq:standardform} and the definition of the $f_t$.
The first part of the definition of a minimal measurable cover now implies that
\begin{equation*}
\left(\sup_{t\in T}\abs*{\hat{\E}^{P}_n(f_t)-\E^{P}(f_t)}\right)^*
\geqslant
\sup_{t\in T}\abs*{\hat{\E}^{P}_n(f_t)-\E^{P}(f_t)}
\geqslant 
\abs*{\Eoh^{\Pee}_n(f)-\Eo^{\Pee}(f)}.
\end{equation*}
Since the left-hand side of this inequality is measurable, the second part of the definition of a minimal measurable cover implies that
\begin{equation*}
\left(\sup_{t\in T}\abs*{\hat{\E}^{P}_n(f_t)-\E^{P}(f_t)}\right)^* 
\geqslant 
\abs*{\Eoh^{\Pee}_n(f)-\Eo^{\Pee}(f)}^*\quad\text{almost surely}. 
\end{equation*}
Combined with \cref{eq:supcon}, we find that
\begin{equation*}
P^{\infty}\left(\lim_{n\to+\infty}\abs*{\Eoh^{\Pee}_n(f)-\Eo^{\Pee}(f)}^*=0\right)=1,
\end{equation*}
because a countable union of null sets is still null.
\end{proofof}

\subsubsection{Proof of Theorem~\ref{th:brack}}
The following proof relies directly on ideas in the proof of~\cite[Thm.~2.4.1]{van1996weak}.

\begin{proofof}{Theorem~\ref{th:brack}}
Fix any~\(\epsilon>0\).
Choose finitely many brackets \((\underline{\phi}_{\epsilon,i},\overline{\phi}_{\epsilon,i})\)\footnote{Brackets are defined in \cref{def:brack}.}, with \(i\) a positive integer index, such that their union contains \(\Phi\) and \(\E^P\group{\overline{\phi}_{\epsilon,i}-\underline{\phi}_{\epsilon,i}}<\epsilon\).
Then, for every \(\phi \in \Phi\), there is a bracket \((\underline{\phi}_{\epsilon,i},\overline{\phi}_{\epsilon,i})\) such that for every positive integer~\(n\)
\begin{equation*}
\hat{\E}^P_n(\phi)-\E^P(\phi) 
\leqslant\hat{\E}^P_n(\overline{\phi}_{\epsilon,i})-\E^P(\overline{\phi}_{\epsilon,i})+\E^P(\overline{\phi}_{\epsilon,i}-\phi) 
\leqslant\hat{\E}^P_n(\overline{\phi}_{\epsilon,i})-\E^P(\overline{\phi}_{\epsilon,i})+\epsilon.
\end{equation*}
As a consequence,
\begin{equation*}
\sup_{\phi\in\Phi}\group*{\hat{\E}^P_n(\phi)-\E^P(\phi)} 
\leqslant\max_i \group*{\hat{\E}^P_n(\overline{\phi}_{\epsilon,i})-\E^P(\overline{\phi}_{\epsilon,i})}+\epsilon 
\leqslant\max_i\abs*{\hat{\E}^P_n(\overline{\phi}_{\epsilon,i})-\E^P(\overline{\phi}_{\epsilon,i})}+\epsilon.
\end{equation*}
By Corollary~\ref{lem:8}, the right-hand side converges to \(\epsilon\).
Using a similar argument we find that
\begin{equation*}
\hat{\E}^P_n(\phi)-\E^P(\phi)
\geqslant \hat{\E}^P_n(\underline{\phi}_{\epsilon,i})-\E^P(\underline{\phi}_{\epsilon,i})-\E^P(\phi-\underline{\phi}_{\epsilon,i}) 
\geqslant \hat{\E}^P_n(\overline{\phi}_{\epsilon,i})-\E^P(\overline{\phi}_{\epsilon,i})-\epsilon
\end{equation*}
and therefore
\begin{equation*}
\inf_{\phi\in\Phi}\group*{\hat{\E}^P_n(\phi)-\E^P(\phi)}
\geqslant\min_i\group*{ \hat{\E}^P_n(\underline{\phi}_{\epsilon,i})-\E^P(\underline{\phi}_{\epsilon,i})}-\epsilon 
\geqslant-\max_i\abs*{ \hat{\E}^P_n(\underline{\phi}_{\epsilon,i})-\E^P(\underline{\phi}_{\epsilon,i})}-\epsilon.
\end{equation*}
The right-hand side converges to \(-\epsilon\) by the same corollary.
Taken together, these two arguments imply that \(\limsup_n \sup_{\phi\in\Phi}\abs*{\hat{\E}^P_n(\phi)-\E^P(\phi)}^*\leqslant \epsilon\) almost surely, for every \(\epsilon>0\).
Hence the \(\limsup\) must actually be zero almost surely.
\end{proofof}

\subsubsection{Proof of Proposition~\ref{prop:finiT}}
% \begin{proofof}{Proposition~\ref{prop:finiT} (old)}
% We start from \cref{th:brack}.
% For any \(\epsilon>0\), we choose the set of brackets \(\FF_\epsilon\coloneqq\cset{(f_t,f_t)}{f_t\in\FF}\).
% It holds trivially that 
% \(f_t\leqslant f_t\leqslant f_t\) and \(\E^P(f_t-f_t)=0<\epsilon\).
% The number of elements in \(\FF_\epsilon\) is equal to the number of elements in \(\FF\), which is in turn equal to the number of elements in \(T\), which is finite, by assumption.
% Then \cref{th:brack} guarantees that \(\FF\) is a Glivenko--Cantelli class with respect to \(P\), so we can apply \cref{th:consistent} to guarantee consistency of the estimator~\(\Eoh_n^\Pee(f)\).
% \end{proofof}

\begin{proofof}{Proposition~\ref{prop:finiT}}
% Rewritten by Gert
Consider the sets \(S_t\coloneqq\cset{\omega\in\reals^\infty}{\lim_{n\to\infty}\hat{\E}_n^{P}(f_t)(\omega)=\E^{P}(f_t)}\), \(t\in T\).
Each \(S_t\) has measure one by the strong law of large numbers, and so does, therefore, their finite intersection $S\coloneqq\bigcap_{t\in T}S_t$.
Fix any \(\omega\) in \(S\).
For any \(\epsilon>0\) and \(t\in T\) there is some natural \(N_{\epsilon,t}\) such that \(\abs{\hat{\E}_n^{P}(f_t)(\omega)-\E^{P}(f_t)}<\epsilon\) for all \(n\geq N_{\epsilon,t}\), and therefore
\begin{equation*}
\abs[\Big]{\hat{\Eo}_n^{\Pee}(f)(\omega)-\Eo^{\Pee}(f)}
=\abs[\bigg]{\min_{t\in T}\hat{\E}_n^{P}(f_t)(\omega)-\min_{t\in T}\E^{P}(f_t)}
\leq\max_{t\in T}\abs[\Big]{\hat{\E}_n^{P}(f_t)(\omega)-\E^{P}(f_t)}<\epsilon
\text{ for all \(n\geq N_\epsilon\coloneqq\max_{t\in T}N_{\epsilon,t}\)},
\end{equation*}
where the first inequality follows from Lemma~\ref{lem:supinf}.
This tells us that \(\hat{\Eo}_n^{\Pee}(f)(\omega)\to\Eo^{\Pee}(f)\) for all \(\omega\in S\), so, indeed, \(\Eoh_n^{\Pee}(f)\) is a strongly consistent estimator, because \(S\) has measure one.
\end{proofof}
Notice how we proved the following corollary in the previous proof:
\begin{corollary}\label{lem:8}
Suppose that \(T\) is finite, then 
\(
\lim_{n\to\infty}\max_{t\in T}\abs[\Big]{\hat{\E}_n^{P}(f_t)-\E^{P}(f_t)}=0 \text{ almost surely.}
\)
\end{corollary}
% \begin{proof}
% Consider the sets \(S_t\coloneqq\cset{\omega\in\reals^\infty}{\lim_{n\to\infty}\hat{\E}_n^{P}(f_t)(\omega)=\E^{P}(f_t)}\), \(t\in T\).
% Each \(S_t\) has measure one by the strong law of large numbers, and so does, therefore, their finite intersection $S\coloneqq\bigcap_{t\in T}S_t$.
% Fix any \(\omega\) in \(S\).
% For any \(\epsilon>0\) and \(t\in T\) there is some natural \(N_{\epsilon,t}\) such that \(\abs{\hat{\E}_n^{P}(f_t)(\omega)-\E^{P}(f_t)}<\epsilon\) for all \(n\geq N_{\epsilon,t}\), and therefore
% \begin{equation*}
% \max_{t\in T}\abs[\Big]{\hat{\E}_n^{P}(f_t)(\omega)-\E^{P}(f_t)}<\epsilon
% \text{ for all \(n\geq N_\epsilon\coloneqq\max_{t\in T}N_{\epsilon,t}\)}.
% \end{equation*}
% This is almost sure convergence, because \(S\) has measure one.
% \end{proof}

\subsubsection{Proof of Theorem~\ref{th:lipschitz}}
The following proof relies directly on ideas in the proof of~\cite[Thm.~2.7.11]{van1996weak}.

\begin{proofof}{Theorem~\ref{th:lipschitz}}
% Checked by Gert
It follows from the definition of \(N(\epsilon,T,d)\) that there are \(N(\epsilon,T,d)\) \(\epsilon\)-balls in \(T\), whose centres we will denote by \(t_k\) for \(k\in\vset{1,\dots,N(\epsilon,T,d)}\), that cover the set \(T\) for \(d\).
We are done if we can prove that the corresponding brackets \([f_{t_k}-\epsilon F,f_{t_k}+\epsilon F]\) for \(k\in\vset{1,\dots,N(\epsilon,T,d)}\) cover \(\FF\).
By assumption, we have that for any \(s\in T\) there is some \(k_s\in\vset{1,\dots,N(\epsilon,T,d)}\) such that \(s\) is contained in the \(\epsilon\)-ball with centre \(t_{k_s}\), implying that \(d(s,t_{k_s})<\epsilon\).
Because of the Lipschitz condition, we then have that for all \(x\in\reals\)
\begin{equation*}
\abs{f_s(x)-f_{t_{k_s}}(x)}\leqslant d(s,t_{k_s})F(x)\leqslant\epsilon F(x),
\end{equation*}
or equivalently
\begin{equation*}
f_{t_{k_s}}(x)-\epsilon F(x)\leqslant f_s(x)\leqslant f_{t_{k_s}}(x)+\epsilon F(x).
\end{equation*}
This means that \(f_s\) is in the \([f_{t_{k_s}}-\epsilon F,f_{t_{k_s}}+\epsilon F]\)-bracket, which is of size \(2\epsilon \norm{F}\).
\end{proofof}

\subsubsection{Proof of Theorem~\ref{th:liptodiff}}

We begin by proving, for the sake of completeness, a well-known lemma, which states that all norms on a finite-dimensional space are equivalent with the \(L_2\)-norm.

\begin{lemma}\label{lem:2norm}
% Checked and simplified by Gert
For any integer \(m>0\) and any norm \(\norm{\cdot}\) and \(L_2\)-norm \(\norm{\cdot}_2\) on \(\reals^m\) there are constants \(a,b>0\) such that
\begin{equation*}
(\forall v\in\reals^m) a \norm{v}\leqslant\norm{v}_2\leqslant b \norm{v}
\end{equation*}
\end{lemma}
\begin{proof}
Let \(\vset{e_1,e_2,\ldots,e_m}\) be any basis for \(\reals^m\). 
Then for any \(v\in\reals^m\) there exist constants \(x_1,x_2,\ldots,x_m\in\reals\) such that \(v=\sum_{k=1}^m x_k e_k\).
By the triangle inequality and the scaling property of norms
\begin{equation*}
\norm{v}\leqslant \sum_{k=1}^m \abs{x_k} \norm{e_k}.
\end{equation*}
We can view this sum as an inner product and apply the Cauchy-Schwarz inequality to find that
\begin{equation*}
\sum_{k=1}^m\abs{x_k}\norm{e_k}\leqslant\group[\bigg]{\sum_{k=1}^m x_k^2}^{\frac12}\group[\bigg]{\sum_{k=1}^m \norm{e_k}^2}^{\frac12}.
\end{equation*}
If we now let \(a\coloneqq\group*{\sum_{k=1}^m\norm{e_k}^2}^{-\frac12}\), we get that, indeed
\begin{equation}\label{eq:linkerkant}
a \norm{v}\leqslant \norm{v}_2.
\end{equation}

For the other inequality, first observe that if \(v=0\) then the theorem is true for any \(b>0\), so we can ignore this case. 
We endow \(R^m\) with the topology induced by the \(L_2\)-norm.
We first prove that the map \(r\colon\reals^m \to\reals_{\geqslant0}\colon v \mapsto \norm{v}\) is continuous for this topology, meaning that for all \(v \in\reals^m\) and for any \(\epsilon>0\) there is some \(\delta>0\) such that for all \(v_1\in\reals^m\) for which \(\norm{v-v_1}_2<\delta\) it holds that \(\abs*{\norm{v}-\norm{v_1}}<\epsilon\).
Indeed, by the reverse triangle inequality and \cref{eq:linkerkant}, 
\begin{equation*}
\abs*{\norm{v}-\norm{v_1}}\leqslant \norm{v-v_1}\leqslant \frac1a \norm{v-v_1}_2,
\end{equation*}
so we can choose \(\delta\coloneqq a\epsilon\).
The set \(J\coloneqq\cset{v\in\reals^m}{\norm{v}_2=1}\) is compact since it is closed and bounded, so the continuous~\(r\) will achieve a minimum \(c>0\) on \(J\).
Hence, \(v/\norm{v}_2\geq c\) for all \(v\in\reals^m\setminus\vset{0}\), and if we therefore let \(b\coloneqq c^{-1}\), then indeed \(\norm{v}_2\leqslant b\norm{v}\).
\end{proof}

We are now ready to prove the theorem.

\begin{proofof}{\cref{th:liptodiff}}
% Checked and improved by Gert
If we start from the result of \cref{th:combo}, then it is enough to prove that the statements (i) that for all \(x\in\reals\) and \(s,t\in T\), \(\norm{\nabla_t f(x,t)}\leqslant F(x)\) and (ii) that \(\E^P(F)<+\infty\), imply the existence of some \(\tilde{F}\colon\reals\to\reals\) for which \(\abs{f(x,s)-f(x,t)}\leqslant\norm{s-t} \tilde{F}(x)\) and \(\E^P(\tilde{F})<+\infty\).
By the fundamental theorem of calculus and the definition of the directional derivative, we can write 
\begin{equation*}
f(x,s)-f(x,t)
=\int_0^1\frac{\mathrm{d}}{\mathrm{d}r}f(x,t+r(s-t))\,\mathrm{d}r
=\int_0^1\nabla_uf(x,u)\big|_{u=t+r(s-t)}\cdot(s-t)\,\mathrm{d}r,
\end{equation*}
where \(\nabla_u\) represents the gradient with respect to \(u \in T_c\).
To bound the absolute value we use the Cauchy-Schwartz inequality:
\begin{align*}
\abs{f(x,s)-f(x,t)}
&=\abs[\bigg]{\int_0^1\nabla_uf(x,u)\big|_{u=t+r(s-t)}\cdot(s-t)\,\mathrm{d}r}
\leqslant\int_0^1\abs[\Big]{\nabla_uf(x,u)\big|_{u=t+r(s-t)}\cdot(s-t)}\,\mathrm{d}r\\
&\leqslant\int_0^1\norm[\Big]{\nabla_uf(x,u)\big|_{u=t+r(s-t)} }_2\norm{s-t}_2\,\mathrm{d}r.
\end{align*}
By Lemma~\ref{lem:2norm}, there is some \(b>0\) such that 
\begin{equation*}
\norm[\Big]{\nabla_uf(x,u)\big|_{u=t+r(s-t)} }_2 
\leqslant b\norm[\Big]{\nabla_u f(x,u)\big|_{u=t+r(s-t)} } 
\leqslant b F(x),
\end{equation*}
where we have also used (i).
Similarly, \(\norm{s-t}_2 \leqslant b \norm{s-t}\), and therefore
\begin{equation*}
\abs{f(x,s)-f(x,t)}\leqslant\int_0^1 b^2 F(x)\norm{s-t}\,\mathrm{d}r
=b^2 F(x)\norm{s-t},
\end{equation*}
If we let \(\tilde{F}\coloneqq b^2 F\) then \(\abs{f(x,s)-f(x,t)}\leqslant\norm{s-t}\tilde{F}(x)\) and \(\E^P(\tilde{F})=b^2\E^P(F)<+\infty\).
\end{proofof}

\subsubsection{Proof of Theorem~\ref{th:compacT}}

\begin{proofof}{\cref{th:compacT}}
% Checked and improved by Gert, one detail missing
Fix any real \(\epsilon>0\).
Let \(g_1(X_1)\coloneqq\sup_{t\in T}\abs*{f_t(X_1)}\).
The (Lebesque) integral associated with \(\E^{P}(g_1)\) can be written as
\begin{equation*}
\E^{P}\group*{g_1}=\int_{\reals}g_1(x)\mathrm{d}P_{X_1}(x),
\end{equation*}
where \(P_{X_1}\) is the probability distribution of the random variable \(X_1\).
The monotone convergence theorem now guarantees that
\begin{equation*}
\E^{P}\group*{g_1}=\lim_{c\to+\infty}\int_{\reals}\mathbb{I}_{[-c,c]}(x)g_1(x)\mathrm{d}P_{X_1}(x)\eqqcolon\int_{-c}^cg_1(x)\mathrm{d}P_{X_1}(x),
\end{equation*}
so there is some \(C>0\) such that
\begin{equation*}
\abs*{\int_{-C}^{C}g_1(x)\mathrm{d}P_{X_1}(x)-\E^{P}\group*{g_1}}=\int_{(-\infty,-C]\cup[C,+\infty)}g_1(x)\mathrm{d}P_{X_1}(x)<\frac{\epsilon}4.
\end{equation*}

Since \(f_t(x)\) is continuously differentiable in \((x,t)\) on an open set that contains the compact set \([-C,C]\times T\), the norm of the gradient will be bounded in \([-C,C]\times T\). 
Because this norm of the gradient is bounded in this set, \(f_t(x)\) will also be Lipschitz continuous on this set, meaning that there is some constant \(K\) for which
\begin{equation*}
\abs*{f_{t_1}(x_1)-f_{t_2}(x_2)}\leqslant K\norm{(x_1-x_2,t_1-t_2)}_2
\text{ for all \(x_1,x_2\in[-C,C]\) and \(t_1,t_2\in T\)}.
\end{equation*}
If we choose \(x=x_1=x_2\), then we find in particular that
\begin{equation*}
\abs*{f_{t_1}(x)-f_{t_2}(x)}\leqslant K\norm{t_1-t_2}_2 
\text{ for all \(x\in[-C,C]\) and \(t_1,t_2\in T\)}.
\end{equation*}

Let \(\delta\coloneqq\frac{\epsilon}{4KP_{X_1}([-C,C])}\) and define open balls in \(T\) with centre \(t\in T\) as \(B_\delta(t)\coloneqq \cset{t_1\in T}{\norm{t_1-t}_2<\delta}\).
Consider the cover \(\cset{B_\delta(t)}{t\in T}\) of \(T\) with \(\delta\)-balls for every \(t\in T\).
Because \(T\) is compact, there is a finite sub-cover, so there is some finite set of centres \(\tau_\delta\subseteq T\) such that \(T\subseteq\bigcup_{t_c\in\tau_\delta}B_\delta(t_c)\).

For any \(t_c\in\tau\), we now define
\begin{equation*}
\underline{\phi[t_c]}_{\epsilon}\colon\reals\to\reals\colon x\mapsto
\inf_{t\in B_\delta(t_c)} f_t(x)
\end{equation*}
and
\begin{equation*}
\overline{\phi[t_c]}_{\epsilon}\colon\reals\to\reals\colon x\mapsto
\sup_{t\in B_\delta(t_c)} f_t(x).
\end{equation*}
By construction, all maps in \(\cset{f_t}{t\in B_\delta(t_c)}\) are inside the bracket \((\underline{\phi[t_c]}_\epsilon,\overline{\phi[t_c]}_\epsilon)\).
Hence, since  \(T\subseteq\bigcup_{t_c\in \tau} B_\delta(t_c)\), each~\(f_t\) belongs to at least one bracket.
Now observe that
\begin{equation*}
\E^P\group*{\overline{\phi[t_c]}_{\epsilon}-\underline{\phi[t_c]}_{\epsilon}} 
=\int_\reals\group[\bigg]{\sup_{t\in B_\delta(t_c)} f_t(x)-\inf_{t\in B_\delta(t_c)} f_t(x)}\mathrm{d}P_{X_1}(x)
=\int_\reals \sup_{s,t\in B_\delta(t_c)}\abs*{ f_t(x)-f_s(x)}\mathrm{d}P_{X_1}(x),
\end{equation*}
where we can take the absolute value on the right-hand side since the supremum of the difference will certainly be positive, because, by symmetry, we could swap \(s\) and \(t\) to flip the sign.
Splitting off the tails leads to:
\begin{equation*}
\int_\reals \sup_{s,t\in B_\delta(t_c)}\abs*{ f_t(x)-f_s(x)}\mathrm{d}P_{X_1}(x) 
\leqslant\int_{-C}^C\sup_{s,t\in B_\delta(t_c)}\abs*{ f_t(x)-f_s(x)}\mathrm{d}P_{X_1}(x)+\int_{(-\infty,-C]\cup[C,+\infty)}\sup_{s,t\in B_\delta(t_c)}\abs*{ f_t(x)-f_s(x)}\mathrm{d}P_{X_1}(x).
\end{equation*}
We will now bound both terms.
For the first term, we get:
\begin{equation*}
\int_{-C}^C \sup_{s,t\in B_\delta(t_c)}\abs*{ f_t(x)-f_s(x)}\mathrm{d}P_{X_1}(x) 
\leqslant K\sup_{s,t\in B_\delta(t_c)}\norm{s-t}_2\int_{-C}^C\mathrm{d}P_{X_1}(x) 
\leqslant 2 K \delta P_{X_1}([-C,C])=\frac{\epsilon}{2}.
\end{equation*}
The second term can be bounded using the results in the first part of the proof.
By the triangle inequality we obtain
\begin{align*}
\int_{(-\infty,-C]\cup[C,+\infty)} \sup_{s,t\in B_\delta(t_c)}\abs*{ f_t(x)-f_s(x)}\mathrm{d}P_{X_1}(x) 
&\leqslant 2 \int_{(-\infty,-C]\cup[C,+\infty)} \sup_{t\in B_\delta(t_c)}\abs*{ f_t(x)} \mathrm{d}P_{X_1}(x)\\
&\leqslant 2 \int_{(-\infty,-C]\cup[C,+\infty)} \sup_{t\in T}\abs*{ f_t(x)} \mathrm{d}P_{X_1}(x) < \frac{\epsilon}{2}.
\end{align*}
We find that \(\E^P(\overline{\phi[t_c]}_\epsilon-\underline{\phi[t_c]}_\epsilon)<\epsilon\). 
We conclude that for every real~\(\epsilon>0\), we can construct a finite set of brackets that satisfies the conditions of \cref{th:brack}. 
\cref{th:brack} therefore guarantees the (strong) consistency of the estimator.
\end{proofof}

\subsubsection{Proof of Theorem~\ref{th:comboimp}}

\begin{proofof}{\cref{th:comboimp}}
% Checked, corrected and improved by Gert
Consider the Borel measurable map~\(\tilde{F}\colon\reals\to\reals\) defined by
\begin{equation*}
\tilde{F}(x)\coloneqq
\begin{cases}
F(x)\frac{\abs{f(x)}}{p(x)}
&\text{ if \(p(x)>0\)}\\
0
&\text{ if \(p(x)=0\)}
\end{cases}
\text{ for all \(x\in\reals\)}.
\end{equation*}
We extend the domain of the \(f_t\) in \cref{eq:functions:for:importance:sampling} by letting
\begin{equation*}
f_t(x)\coloneqq
\begin{cases}
f(x)\frac{p_t(x)}{p(x)}
&\text{ if \(p(x)>0\)}\\
0
&\text{ if \(p(x)=0\)}
\end{cases}
\text{ for all \(x\in\reals\)}.
\end{equation*}
It follows from the assumptions that
\begin{equation*}
\abs{p_s(x)-p_t(x)}\leqslant\norm{s-t}F(x)\text{ for all \(s,t\in T\) and \(x\in\reals\),}
\end{equation*}
and if we multiply both sides by \(\frac{\abs*{f(x)}}{p(x)}\), then we get for all \(s,t\in T\) that
\begin{equation*}
\frac{\abs*{f(x)}}{p(x)}\abs{p_s(x)-p_t(x)}\leqslant\norm{s-t}\tilde{F}(x)
\text{ for all real \(x\) such that \(p(x)>0\)},
\end{equation*}
and taking into account the definitions of \(\tilde{F}\), \(f_s\) and \(f_t\) above, this leads to
\begin{equation*}
\abs{f_s(x)-f_t(x)}\leqslant\norm{s-t}\tilde{F}(x)\text{ for all \(s,t\in T\) and \(x\in\reals\)}.
\end{equation*}
Furthermore, it follows from the assumptions that
\begin{equation*}
\E^P(\tilde{F})
=\int_\reals \tilde{F}(x) p(x)\dx
\leqslant\int_\reals \abs{f(x)} F(x)\dx 
<+\infty.
\end{equation*}
These last two inequalities show that \(\tilde{F}\) is a map that satisfies the requirements of the map~\(F\) of \cref{th:combo}.
Applying this theorem yields the required result.
\end{proofof}

\subsubsection{Proof of Theorem~\ref{th:liptodiffimp}}

\begin{proofof}{\cref{th:liptodiffimp}}
% Checked, corrected and improved by Gert
Consider the Borel measurable map~\(\tilde{F}\colon\reals\to\reals\) defined by
\begin{equation*}
\tilde{F}(x)\coloneqq
\begin{cases}
F(x)\frac{\abs{f(x)}}{p(x)}
&\text{ if \(p(x)>0\)}\\
0
&\text{ if \(p(x)=0\)}
\end{cases}
\text{ for all \(x\in\reals\)}.
\end{equation*}
We want to apply \cref{eq:functions:for:importance:sampling2}, and to this end define \(f_{T_c}\) as
\begin{equation*}
f_{T_c}(x)\coloneqq
\begin{cases}
f(x)\frac{p_{T_c}(x)}{p(x)}
&\text{ if \(p(x)>0\)}\\
0
&\text{ if \(p(x)=0\)}
\end{cases}
\text{ for all \(x\in\reals\)}.
\end{equation*}
It follows from the assumptions that

\begin{equation*}
\norm{\nabla_t p_{T_c}(x,t)}\leqslant F(x)\text{ for all \(x\in\reals\) and \(t\in T_c\)},
\end{equation*} 
so, if we multiply both sides of this inequality by \(\frac{\abs*{f(x)}}{p(x)}\), we get
\begin{equation*}
\norm{\nabla_t f_{T_c}(x,t)}_2
=\frac{\abs*{f(x)}}{p(x)}\norm{\nabla_t p_{T_c}(x,t)}_2\leqslant\tilde{F}(x)
\text{ for all \(t\in T_c\) and all real \(x\) for which \(p(x)>0\)},
\end{equation*}
and, taking into account the definitions of \(f_{T_c}\) and \(\tilde{F}\) given above, this finally leads to 
\begin{equation*}
\norm{\nabla_t f_{T_c}(x,t)}_2\leqslant\tilde{F}(x)
\text{ for all \(t\in T_c\) and all real \(x\)}.
\end{equation*}
Furthermore,
\begin{equation*}
\E^P(\tilde{F})
=\int_\reals \tilde{F}(x) p(x)\dx
\leqslant\int_\reals \abs{f(x)} F(x)\dx 
<+\infty.
\end{equation*}
These last two inequalities show that~\(\tilde{F}\) is a map that satisfies the requirements of the map~\(F\) of \cref{th:liptodiff}.
Applying this theorem yields the required result.
\end{proofof}

\subsubsection{Proof of Theorem~\ref{th:compacTimp}}

\begin{proofof}{\cref{th:compacTimp}}
% Checked and corrected and improved by Gert, but still one point needs attention
Fix any real~\(\epsilon>0\).
Since \(f\) is bounded, there is some \(M>0\) such that \((\forall x\in\reals)\abs*{f(x)}<M\).
We start from the integral condition in the assumptions:
\begin{equation*}
\int_{\reals}\sup_{t\in T}p_t(x)\dx<+\infty.
\end{equation*}
The monotone convergence theorem now guarantees that
\begin{equation*}
\int_{\reals}\sup_{t\in T}p_t(x)\dx=\lim_{c\to+\infty}\int_{\reals}\mathbb{I}_{[-c,c]}(x)\sup_{t\in T}p_t(x)\dx\eqqcolon\lim_{c\to+\infty}\int_{-c}^c \sup_{t\in T}p_t(x)\dx,
\end{equation*}
so there is some real~\(C>0\) such that
\begin{equation*}
\abs*{\int_{-C}^{C}\sup_{t\in T}p_t(x)\dx-\int_{\reals}\sup_{t\in T}p_t(x)\dx}
=\int_{(-\infty,-C]\cup[C,+\infty)}\sup_{t\in T}p_t(x)\dx 
<\frac{\epsilon}{4 M}.
\end{equation*}

Since \(p_{T_c}(x,t)\) is continuously differentiable in~\((x,t)\) on a set that includes \([-C,C]\times T\), a compact set, it will also be Lipschitz continuous on this set, meaning that there is a constant~\(K\) such that
\begin{equation*}
\abs*{p_{t_1}(x_1)-p_{t_2}(x_2)}\leqslant K\norm{(x_1-x_2,t_1-t_2)}_2
\text{ for all \(x_1,x_2\in[-C,C]\) and \(t_1,t_2\in T\)}.
\end{equation*}
If we choose \(x=x_1=x_2\), then we find in particular that
\begin{equation*}
\abs*{p_{t_1}(x)-p_{t_2}(x)}\leqslant K\norm{t_1-t_2}_2
\text{ for all \(x\in[-C,C]\) and \(t_1,t_2\in T\)}.
\end{equation*}

Let \(\delta\coloneqq \frac{\epsilon}{4 C K M}\) and define open balls in \(T\) with centre \(t\in T\) as \(B_\delta(t)\coloneqq \cset{t_1\in T}{\norm{t-t_1}_2 <\delta}\).
Consider the cover~\(\cset{B_\delta(t)}{t\in T}\) of~\(T\) with \(\delta\)-balls for every \(t\in T\).
Because \(T\) is compact, there is a finite sub-cover, so there is a finite set of centres \(\tau_\delta\subseteq T\) such that \(T\subseteq\bigcup_{t_c\in\tau_\delta}B_\delta(t_c)\).

For any \(t_c\in\tau_\delta\), we now define
\begin{equation*}
\underline{\phi[t_c]}_{\epsilon}\colon\reals\to\reals\colon x\mapsto
\inf_{t\in B_\delta(t_c)}f(x)\frac{p_t(x)}{p(x)}
\end{equation*}
and
\begin{equation*}
\overline{\phi[t_c]}_{\epsilon}\colon\reals\to\reals\colon x\mapsto
\sup_{t\in B_\delta(t_c)}f(x)\frac{p_t(x)}{p(x)}.
\end{equation*}
By construction, all maps in~\(\cset{f_t}{t\in B_\delta(t_c)}\) 
are inside the bracket~\((\underline{\phi[t_c]}_\epsilon,\overline{\phi[t_c]}_\epsilon)\).
Hence, since  \(T\subseteq\bigcup_{t_c\in \tau}B_\delta(t_c)\), each~\(f_t\) belongs to at least one bracket.
Now observe that
\begin{equation*}
\E^P \group*{\overline{\phi[t_c]}_{\epsilon}-\underline{\phi[t_c]}_{\epsilon}} 
=\int_\reals\group[\bigg]{\sup_{t\in B_\delta(t_c)}f(x)\frac{p_t(x)}{p(x)}-\inf_{t\in B_\delta(t_c)}f(x)\frac{p_t(x)}{p(x)}}p(x)\dx 
=\int_\reals \abs{f(x)}\sup_{s,t\in B_\delta(t_c)}\abs*{ p_t(x)-p_s(x)}\dx,
\end{equation*}
where we can take the absolute value on the right-hand side since the supremum of the difference will certainly be positive, because, by symmetry, we could swap \(s\) and \(t\) to flip the sign.
Splitting off the tails leads to:
\begin{equation*}
\int_\reals \abs{f(x)}\sup_{s,t\in B_\delta(t_c)}\abs*{ p_t(x)-p_s(x)}\dx \leqslant \int_{-C}^C \abs*{f(x)} \sup_{s,t\in B_\delta(t_c)}\abs*{ p_t(x)-p_s(x)}\dx + \int_{(-\infty,-C]\cup[C,+\infty)} \abs*{f(x)} \sup_{s,t\in B_\delta(t_c)}\abs*{ p_t(x)-p_s(x)}\dx.
\end{equation*}
We will bound both terms.
The first term  can be bounded by
\begin{equation*}
\int_{-C}^C\abs*{f(x)}\sup_{s,t\in B_\delta(t_c)}\abs*{ p_t(x)-p_s(x)}\dx 
\leqslant KM\sup_{s,t\in B_\delta(t_c)}\norm{s-t}_2\int_{-C}^C\dx 
\leqslant 2CKM\delta 
=\frac{\epsilon}{2}.
\end{equation*}
The second term can be bounded using the results in the first part of the proof.
By the triangle inequality we obtain
\begin{align*}
\int_{(-\infty,-C]\cup[C,+\infty)}\abs*{f(x)}\sup_{s,t\in B_\delta(t_c)}\abs*{p_t(x)-p_s(x)}\dx 
&\leqslant 2M\int_{(-\infty,-C]\cup[C,+\infty)}\sup_{t\in B_\delta(t_c)}\abs*{p_t(x)}\dx  
<\frac{\epsilon}{2}.
\end{align*}
Hence, \(\E^P(\overline{\phi[t_c]}_\epsilon-\underline{\phi[t_c]}_\epsilon)<\epsilon\). 
We conclude that for every real \(\epsilon>0\), we can construct a finite set of brackets that satisfies the conditions of \cref{th:brack}. 
\cref{th:brack} therefore guarantees consistency of the estimator.
\end{proofof}

\subsubsection{The map \texorpdfstring{\(F\)}{\it F} for the first example using Theorem~\ref{th:liptodiff}}
First of all, we let \(z\coloneqq\frac{x-\mu}{\sigma}\).
We have to find an upper bound \(F(x)\) for all \(\mu\) and \(\sigma\) for \(\norm*{\nabla_{(\mu,\sigma)} f_{(\mu,\sigma)}(x)}_2\).
We find that
\begin{align*}
\frac{\partial}{\partial\mu}f_{(\mu,\sigma)}(x) 
&=\frac{(x-\mu)}{\sigma^2}\mathbb{I}_{\cset{y}{g(y)<0}}(x) 
\frac{\sigma_o}{\sigma}\exp\group[\Big]{-\frac{(x-\mu)^2}{2\sigma^2}+\frac{(x-\mu_o)^2}{2\sigma_o^2}}
=\frac{z}{\sigma}f_{\mu,\sigma}(x)\\
\frac{\partial}{\partial\sigma}f_{(\mu,\sigma)}(x) 
&=\group[\Big]{\frac{(x-\mu)^2}{\sigma^3}-\frac1\sigma}\mathbb{I}_{\cset{y}{g(y)<0}}(x) 
\frac{\sigma_o}{\sigma}\exp\group[\Big]{-\frac{(x-\mu)^2}{2\sigma^2}+\frac{(x-\mu_o)^2}{2\sigma_o^2}}
=\frac{(z^2-1)}{\sigma}f_{\mu,\sigma}(x).
\end{align*}
Taking the norm results in
\begin{equation*}
\norm*{\nabla_{(\mu,\sigma)}f_{(\mu,\sigma)}(x)}_2
=\frac1{\sigma}\sqrt{z^2+(z^2-1)^2}f_{\mu,\sigma}(x)
=\frac1{\sigma}f_{\mu,\sigma}(x)\sqrt{z^4-z^2+1}.
\end{equation*}
Now we can bound each factor separately:
\begin{align*}
\frac1{\sigma}
&\leqslant\frac1{\sigmu}
\\
f_{\mu,\sigma}(x) 
&=\mathbb{I}_{\cset{y}{g(y)>0}}(x)
\frac{\sigma_o}{\sigma}
\exp\group[\Big]{-\frac{(x-\mu)^2}{2\sigma^2}+\frac{(x-\mu_o)^2}{2\sigma_o^2}}
\leqslant\frac{\sigma_o}{\sigmu}
\begin{cases}
\exp\group[\Big]{-\frac{(x-\muu)^2}{2\sigmo^2}+\frac{(x-\mu_o)^2}{2\sigma_o^2}}
&x<\muu\\
\exp\group[\Big]{\frac{(x-\mu_o)^2}{2\sigma_o^2}}
& \muu\leqslant x<\muo\\
\exp\group[\Big]{-\frac{(x-\muo)^2}{2\sigmo^2}+\frac{(x-\mu_o)^2}{2\sigma_o^2}}
&\muo\leqslant x
\end{cases}
\\
\sqrt{z^4-z^2+1} 
&\leqslant\sqrt{z^4+2z^2+1} 
=z^2+1
=\group[\Big]{\frac{x-\mu}{\sigma}}^2+1
\leqslant 
\begin{cases}
\frac{(x-\muo)^2}{\sigmu^2} +1
&x<\muu\\
\frac{(\muo-\muu)^2}{\sigmu^2}+1
&\muu\leqslant x<\muo\\
\frac{(x-\muu)^2}{\sigmu^2}+1
&\muo\leqslant x.
\end{cases}
\end{align*}
So we suggest
\begin{equation*}
F(x)
=\frac{\sigma_o}{\sigmu^2}
\begin{cases}
\group[\Big]{\frac{(x-\muo)^2}{\sigmu^2}+1}
\exp\group[\Big]{-\frac{(x-\muu)^2}{2\sigmo^2}+\frac{(x-\mu_o)^2}{2\sigma_o^2}}
&x<\muu\\
\group[\Big]{\frac{(\muo-\muu)^2}{\sigmu^2}+1}
\exp\group[\Big]{\frac{(x-\mu_o)^2}{2\sigma_o^2}}
&\muu\leqslant x<\muo\\
\group[\Big]{\frac{(x-\muu)^2}{\sigmu^2}+1}
\exp\group[\Big]{-\frac{(x-\muo)^2}{2\sigmo^2}+\frac{(x-\mu_o)^2}{2\sigma_o^2}}
&\muo\leqslant x.
\end{cases}
\end{equation*}
Now we check whether \(\E^{P_o}(\abs{F})<+\infty\).
\begin{align*}
\E^{P_o}(\abs{F})
&=\int_{\reals}\abs{F(x)}\frac1{\sqrt{2\pi}\sigma_o}
\exp\group[\bigg]{-\frac{(x-\mu_o)^2}{2\sigma_o^2}}\dx\\
&=\frac1{\sqrt{2\pi}\sigmu^2}
\vset[\bigg]{
\int_{-\infty}^{\muu}\group[\bigg]{\frac{(x-\muo)^2}{\sigmu^2}+1}
\exp\group[\bigg]{-\frac{(x-\muu)^2}{2\sigmo^2}}\dx
+\int_{\muu}^{\muo}\group[\bigg]{\frac{(\muo-\muu)^2}{\sigmu^2}+1}\dx\\
&\hspace{6cm}+\int_{\muo}^{+\infty}\group[\bigg]{\frac{(x-\muu)^2}{\sigmu^2}+1}
\exp\group[\bigg]{-\frac{(x-\muo)^2}{2\sigmo^2}}\dx
}\\
&\leqslant
\frac{\sigmo}{\sigmu^2}\E^{P_{(\muu,\sigmo)}}\group[\bigg]{\frac{(X-\muo)^2}{\sigmu^2}+1}
+\frac1{\sqrt{2\pi}\sigmu^2}\group[\bigg]{\frac{(\muo-\muu)^2}{\sigmu^2}+1}\group[\big]{\muo-\muu}
+\frac{\sigmo}{\sigmu^2}\E^{P_{(\muo,\sigmo)}}\group[\bigg]{\frac{(X-\muu)^2}{\sigmu^2}+1}.
\end{align*}
Since the expectations are taken over polynomials, and normal distributions have finite moments, we know that this will be finite.

% The bound follows from straightforward estimation.

% \begin{align*}
% \norm*{\nabla_{(\mu,\sigma)} f_{(\mu,\sigma)}(x)}_2 &= f_{(\mu,\sigma)}(x) \frac{\sqrt{((x-\mu)^2-\sigma^2)^2+\sigma^2(x-\mu)^2}}{\sigma^3} \\
% &\leqslant f_{(\mu,\sigma)}(x) \frac{\abs*{x-\mu}}{\sigma^3}\sqrt{(x-\mu)^2+\sigma^2}\\
% &\leqslant
% \begin{cases}
% f_{(\muu,\sigmo)}(x) \sigmo \frac{\muo-x}{\sigmu^4} \sqrt{\sigmo^2+(x-\muo)^2}, & x<\muu,\\
% \mathbb{I}_{\cset{y}{g(y)<0}}(x) \sigma_o \exp \left( \frac{(x-\mu_o)^2}{2\sigma_o^2} \right)\frac{\muo-\muu}{\sigmu^4} \sqrt{\sigmo^2+(\muo-\muu)^2}, & \muu \leqslant x \leqslant \muo,\\
% f_{(\muo,\sigmo)}(x) \sigmo \frac{x-\muu}{\sigmu^4} \sqrt{\sigmo^2+(x-\muu)^2}, &  x>\muo\\
% \end{cases}\\
% &=F(x).
% \end{align*}

% Proving that the expected value is finite is again straightforward.

% \begin{align*}
% \E(F(X))&=\int_{\reals} F(x) \frac1{\sqrt{2\pi}\sigma_o}\exp \left( -\frac{(x-\mu_o)^2}{2\sigma_o^2}\right)\dx\\
% &=\frac1{\sqrt{2\pi}}\int_{-\infty}^{\muu} \mathbb{I}_{\cset{y}{g(y)<0}}(x) \exp \left(-\frac{(x-\muu)^2}{2 \sigmo^2} \right)\frac{\muo-x}{\sigmu^4} \sqrt{\sigmo^2+(x-\muo)^2}\\
% &\qquad +\frac1{\sqrt{2\pi}}\int_{\muu}^{\muo} \mathbb{I}_{\cset{y}{g(y)<0}}(x) \frac{(\muo-\muu)}{\sigmu^4} \sqrt{\sigmo^2+(\muo-\muu)^2}\\
% &\qquad +\frac1{\sqrt{2\pi}}\int_{\muo}^{+\infty} \mathbb{I}_{\cset{y}{g(y)<0}}(x) \exp \left(-\frac{(x-\muo)^2}{2 \sigmo^2} \right) \frac{x-\muu}{\sigmu^4} \sqrt{\sigmo^2+(x-\muu)^2}\\
% &< +\infty.
% \end{align*}

\subsubsection{The map \texorpdfstring{\(F\)}{\it F} for the first example using Theorem~\ref{th:liptodiffimp}}
This derivation is similar to the previous derivation.
First of all, we let \(z\coloneqq\frac{x-\mu}{\sigma}\).
We have to find an upper bound \(F(x)\) for all \(\mu\) and \(\sigma\) for \(\norm*{\nabla_{(\mu,\sigma)} p_{(\mu,\sigma)}(x)}_2\).
We find that
\begin{align*}
\frac{\partial}{\partial\mu}p_{(\mu,\sigma)}(x) 
&=\frac{(x-\mu)}{\sigma^2}
\frac{1}{\sqrt{2\pi}\sigma}\exp\group[\Big]{-\frac{(x-\mu)^2}{2\sigma^2}}
=\frac{z}{\sigma}p_{\mu,\sigma}(x)\\
\frac{\partial}{\partial\sigma}p_{(\mu,\sigma)}(x) 
&=\group[\Big]{\frac{(x-\mu)^2}{\sigma^3}-\frac1\sigma}
\frac{1}{\sqrt{2\pi}\sigma}\exp\group[\Big]{-\frac{(x-\mu)^2}{2\sigma^2}}
=\frac{(z^2-1)}{\sigma}p_{\mu,\sigma}(x).
\end{align*}
Taking the norm results in
\begin{equation*}
\norm*{\nabla_{(\mu,\sigma)}p_{(\mu,\sigma)}(x)}_2
=\frac1{\sigma}\sqrt{z^2+(z^2-1)^2}p_{\mu,\sigma}(x)
=\frac1{\sigma}p_{\mu,\sigma}(x)\sqrt{z^4-z^2+1}.
\end{equation*}
Now we can bound each factor separately:
\begin{align*}
\frac1{\sigma}
&\leqslant\frac1{\sigmu}
\\
p_{\mu,\sigma}(x) 
&=
\frac{1}{\sqrt{2\pi}\sigma}
\exp\group[\Big]{-\frac{(x-\mu)^2}{2\sigma^2}}
\leqslant\frac{1}{\sqrt{2\pi}\sigmu}
\begin{cases}
\exp\group[\Big]{-\frac{(x-\muu)^2}{2\sigmo^2}}
&x<\muu\\
1
& \muu\leqslant x<\muo\\
\exp\group[\Big]{-\frac{(x-\muo)^2}{2\sigmo^2}}
&\muo\leqslant x
\end{cases}
\\
\sqrt{z^4-z^2+1} 
&\leqslant\sqrt{z^4+2z^2+1} 
=z^2+1
=\group[\Big]{\frac{x-\mu}{\sigma}}^2+1
\leqslant 
\begin{cases}
\frac{(x-\muo)^2}{\sigmu^2} +1
&x<\muu\\
\frac{(\muo-\muu)^2}{\sigmu^2}+1
&\muu\leqslant x<\muo\\
\frac{(x-\muu)^2}{\sigmu^2}+1
&\muo\leqslant x.
\end{cases}
\end{align*}
So we suggest
\begin{equation*}
F(x)
=\frac{1}{\sqrt{2\pi}\sigmu^2}
\begin{cases}
\group[\Big]{\frac{(x-\muo)^2}{\sigmu^2}+1}
\exp\group[\Big]{-\frac{(x-\muu)^2}{2\sigmo^2}}
&x<\muu\\
\frac{(\muo-\muu)^2}{\sigmu^2}+1
&\muu\leqslant x<\muo\\
\group[\Big]{\frac{(x-\muu)^2}{\sigmu^2}+1}
\exp\group[\Big]{-\frac{(x-\muo)^2}{2\sigmo^2}}
&\muo\leqslant x.
\end{cases}
\end{equation*}
Now we check whether \(\int_\reals \abs*{f(x)}F(x)\dx<+\infty\).
Since \(\abs{f(x)}\leqslant 1\), it is sufficient to prove that \(\int_\reals F(x)\dx<+\infty\).
\begin{align*}
\int_\reals \abs*{f(x)}F(x)
&\leqslant\frac1{\sqrt{2\pi}\sigmu^2}
\vset[\bigg]{
\int_{-\infty}^{\muu}\group[\bigg]{\frac{(x-\muo)^2}{\sigmu^2}+1}
\exp\group[\bigg]{-\frac{(x-\muu)^2}{2\sigmo^2}}\dx
+\int_{\muu}^{\muo}\group[\bigg]{\frac{(\muo-\muu)^2}{\sigmu^2}+1}\dx\\
&\hspace{6cm}+\int_{\muo}^{+\infty}\group[\bigg]{\frac{(x-\muu)^2}{\sigmu^2}+1}
\exp\group[\bigg]{-\frac{(x-\muo)^2}{2\sigmo^2}}\dx
}\\
&\leqslant
\frac{\sigmo}{\sigmu^2}\E^{P_{(\muu,\sigmo)}}\group[\bigg]{\frac{(X-\muo)^2}{\sigmu^2}+1}
+\frac1{\sqrt{2\pi}\sigmu^2}\group[\bigg]{\frac{(\muo-\muu)^2}{\sigmu^2}+1}\group[\big]{\muo-\muu}
+\frac{\sigmo}{\sigmu^2}\E^{P_{(\muo,\sigmo)}}\group[\bigg]{\frac{(X-\muu)^2}{\sigmu^2}+1}.
\end{align*}
Since the expectations are taken over polynomials, and normal distributions have finite moments, we know that this will be finite.

\subsubsection{The first example for a compact set of parameters using Theorem~\ref{th:compacTimp}}
\begin{align*}
\int_{\reals} \sup_{(\mu,\sigma) \in[\muo,\muu]\times[\sigmo,\sigmu]} p_{(\mu,\sigma)}(x)\dx &\leqslant 
\frac1{\sqrt{2\pi}\sigmu} \int_{\reals} \sup_{\mu \in [\muu,\muo]} e^{-\frac{(x-\mu)^2}{2\sigmo^2}}\dx\\
&=\frac1{\sqrt{2\pi}\sigmu} \group*{\int_{-\infty}^{\muu} e^{-\frac{(x-\muu)^2}{2\sigmo^2}}\dx+\int_{\muu}^{\muo} 1\dx + \int_{\muo}^{+\infty} e^{-\frac{(x-\muo)^2}{2\sigmo^2}}\dx}\\
&=\frac{1}{\sigmu}\group*{\sigmo+\frac1{\sqrt{2\pi}}(\muo-\muu)}<+\infty.
\end{align*}
}{}

\end{document}